\pdfoutput=1
\documentclass[a4paper,11pt,oneside]{amsart}
\usepackage[paperheight=279mm,paperwidth=18cm,textheight=26cm,textwidth=14cm]{geometry}
\usepackage{amsmath,amsfonts,amssymb,amsthm,mathtools}
\usepackage[utf8]{inputenc}
\usepackage[T1]{fontenc}

\usepackage[unicode,hidelinks,bookmarks]{hyperref}
\hypersetup{
colorlinks=true,
citecolor=[rgb]{0,0,0.5},
linkcolor=[rgb]{0,0,0.5},
urlcolor=[rgb]{0,0,0.75},
pdfpagemode=UseNone,
pdfstartview=FitH,
pdfdisplaydoctitle=true,
pdftitle={A nilpotent IP polynomial multiple recurrence theorem},
pdfauthor={Pavel Zorin-Kranich},
pdflang=en-US
}

\usepackage[style=alphabetic]{biblatex}
\numberwithin{equation}{section}
\newtheorem{theorem}[equation]{Theorem}
\newtheorem{lemma}[equation]{Lemma}
\newtheorem{proposition}[equation]{Proposition}
\newtheorem{corollary}[equation]{Corollary}
\theoremstyle{definition}
\newtheorem{definition}[equation]{Definition}
\theoremstyle{remark}
\newtheorem{remark}[equation]{Remark}
\newtheorem{example}[equation]{Example}

\DeclarePairedDelimiter\abs{\lvert}{\rvert}

\DeclarePairedDelimiter\norm{\lVert}{\rVert}

\DeclarePairedDelimiterX\Set[1]\{\}{%

#1
}

\newcommand*{\N}{\mathbb{N}}
\newcommand*{\Z}{\mathbb{Z}}
\newcommand*{\R}{\mathbb{R}}
\newcommand*{\dif}{\mathrm{d}}
\newcommand*{\inv}{^{-1}}
\newcommand*{\dom}{\mathrm{dom}\,}
\newcommand*{\E}{\mathbb{E}}
\newcommand*{\from}{\colon}
\def\<{\left\langle}
\def\>{\right\rangle}
\newcommand*{\Fin}{\mathcal{F}}
\newcommand*{\Fine}{\Fin_{\emptyset}}
\newcommand*{\gb}{g_{\bullet}}
\newcommand*{\hb}{h_{\bullet}}
\newcommand*{\Gb}[1][]{G_{\bullet #1}}
\newcommand*{\Fb}[1][]{F_{\bullet #1}}
\newcommand*{\Wb}[1][]{W_{\bullet #1}}
\newcommand*{\poly}[1][\Gb]{\mathrm{P}(\Gamma,#1)}
\newcommand*{\polyn}[1][\Gb]{\mathrm{VIP}(#1)}
\newcommand*{\PE}[2]{{#1}^{\otimes #2}}
\newcommand*{\FE}{\mathrm{FE}}
\newcommand*{\AP}{\mathrm{AP}}
\DeclareMathOperator*{\IPlim}{IP-lim}
\DeclareMathOperator*{\wIPlim}{w-IP-lim}
\DeclareMathOperator{\im}{im}
\DeclareMathOperator{\fix}{fix}
\DeclareMathOperator{\rank}{rank}
\DeclareMathOperator{\lin}{lin}
\newcommand*{\sD}{\tilde D} 
\newcommand*{\rD}{\hat D}
\DeclarePairedDelimiter\nint{\lfloor}{\rceil}
\DeclarePairedDelimiter\dint{\lVert}{\rVert}
\DeclarePairedDelimiter\floor{\lfloor}{\rfloor}

\newcommand{\calA}{\mathcal{A}}
\newcommand{\calB}{\mathcal{B}}
\newcommand{\calC}{\mathcal{C}}
\newcommand{\calG}{\mathcal{G}}
\newcommand{\calT}{\mathcal{T}}

\begin{document}
\title[Nilpotent IP polynomial multiple recurrence]{A nilpotent IP polynomial\\ multiple recurrence theorem}
\author{Pavel Zorin-Kranich}
\address{University of Amsterdam\\
Korteweg-de Vries Institute for Mathematics}
\begin{abstract}
We generalize the IP-polynomial Szemer\'edi theorem due to Bergelson and McCutcheon and the nilpotent Szemer\'edi theorem due to Leibman.
Important tools in our proof include a generalization of Leibman's result that polynomial mappings into a nilpotent group form a group and a multiparameter version of the nilpotent Hales--Jewett theorem due to Bergelson and Leibman.
\end{abstract}
\maketitle

\section{Introduction}
Furstenberg's ergodic theoretic proof \cite{MR0498471} of Szemer\'edi's theorem on arithmetic progressions \cite{MR0369312} has led to various generalizations of the latter.
Recall that Furstenberg's original multiple recurrence theorem provides a syndetic set of return times.
The \emph{IP} recurrence theorem of Furstenberg and Katznelson \cite{MR833409}, among other things, improves this to an IP* set.
The idea to consider the limit behavior of a multicorrelation sequence not along a F\o{}lner sequence but along an IP-ring has proved to be very fruitful and allowed them to obtain the density Hales--Jewett theorem \cite{MR1191743}.

In a different direction, Bergelson and Leibman \cite{MR1325795} have proved a \emph{polynomial} multiple recurrence theorem.
That result has been extended from commutative to nilpotent groups of transformations by Leibman \cite{MR1650102}.
Many of the additional difficulties involved in the nilpotent extension were algebraic in nature and have led Leibman to develop a general theory of polynomial mappings into nilpotent groups \cite{MR1910931}.
An important aspect of the proofs of these polynomial recurrence theorems, being present in all later extensions including the present article, is that the induction process involves ``multiparameter'' recurrence even if one is ultimately only interested in the ``one-parameter'' case.

More recently an effort has been undertaken to combine these two directions.
Building on their earlier joint work with Furstenberg \cite{MR1417769}, Bergelson and McCutcheon \cite{MR1692634} have shown the set of return times in the polynomial multiple recurrence theorem is IP*.
Joint extensions of their result and the IP recurrence theorem of Furstenberg and Katznelson have been obtained by Bergelson, H{\aa}land Knutson and McCutcheon for single recurrence \cite{MR2246589} and McCutcheon for multiple recurrence \cite{MR2151599}.
The results of the last two papers also provide multiple recurrence along \emph{admissible generalized polynomials} (Definition~\ref{def:generalized-poly}), a class that includes ordinary polynomials that vanish at the origin, and, more generally, along FVIP-systems (Definition~\ref{def:fvip}), a generalization of IP-systems.

In this article we continue this line of investigation.
Our main result, Theorem~\ref{thm:SZ-general}, generalizes McCutcheon's IP polynomial multiple recurrence theorem to the nilpotent setting.
The content of Theorem~\ref{thm:SZ-general} is best illustrated by the following generalization of Leibman's nilpotent multiple recurrence theorem (here and throughout the article group actions on topological spaces and measure spaces are on the right and on function spaces on the left.).
\begin{theorem}
\label{thm:SZ-intro}
Let $T_{1},\dots,T_{t}$ be invertible measure-preserving transformations on a probability space $(X,\calA,\mu)$ that generate a nilpotent group.
Then for every $A\in\calA$ with $\mu(A)>0$, every $m\in\N$, and any admissible generalized polynomials $p_{i,j}:\Z^{m}\to\Z$, $i=1,\dots,t$, $j=1,\dots,s$, the set
\begin{equation}
\label{eq:SZ-intro}
\Big\{ \vec n\in\Z^{m} : \mu\big( \bigcap_{j=1}^{s} A \big(\prod_{i=1}^{t}T_{i}^{p_{i,j}(\vec n)} \big)\inv \big) > 0 \Big\}
\end{equation}
is FVIP* in $\Z^{m}$, that is, it has nontrivial intersection with every FVIP-system in $\Z^{m}$.
\end{theorem}
In particular, the set \eqref{eq:SZ-intro} is IP*, so that it is syndetic \cite[Lemma 9.2]{MR603625}.

\tableofcontents

\section{Polynomial mappings into nilpotent groups}
\label{sec:poly}
In this section we set up the algebraic framework for dealing with IP-polynomials in several variables with values in a nilpotent group.

We begin with a generalization of Leibman's result that polynomial mappings into a nilpotent group form a group under pointwise operations \cite[Proposition 3.7 and erratum]{MR1910931}.
Following an idea from the proof of that result by Green and Tao \cite[Proposition 6.5]{MR2877065} we encode the information that is contained in Leibman's vector degree in a prefiltration indexed by $\N=\{0,1,\dots\}$ (see \cite[Appendix B]{MR2950773} for related results regarding prefiltrations indexed by more general partially ordered semigroups).

A \emph{prefiltration} $\Gb$ is a sequence of nested groups
\begin{equation}
\label{eq:prefiltration}
G_{0}\geq G_{1} \geq G_{2} \geq \dots
\quad\text{such that}\quad
[G_{i},G_{j}]\subset G_{i+j}
\quad\text{for every }i,j\in\N.
\end{equation}
A \emph{filtration} (on a group $G$) is a prefiltration in which $G_{0}=G_{1}$ (and $G_{0}=G$).
We will frequently write $G$ instead of $G_{0}$.
Conversely, most groups $G$ that we consider in this article are endowed with a prefiltration $\Gb$ such that $G_{0}=G$.
A group may admit several prefiltrations, and we usually fix one of them even if we do not refer to it explicitly.

A prefiltration is said to have \emph{length} $d\in\N$ if $G_{d+1}$ is the trivial group and length $-\infty$ if $G_{0}$ is the trivial group.
Arithmetic for lengths is defined in the same way as conventionally done for degrees of polynomials, i.e.\ $d-t=-\infty$ if $d<t$.

Let $\Gb$ be a prefiltration of length $d$ and let $t\in\N$ be arbitrary.
We denote by $\Gb[+t]$ the prefiltration of length $d-t$ given by $(\Gb[+t])_{i}=G_{i+t}$ and by $\Gb[/t]$ the prefiltration of length $\min(d,t-1)$ given by $G_{i/t}=G_{i}/G_{t}$ (this is understood to be the trivial group for $i\geq t$; note that $G_{t}$ is normal in each $G_{i}$ for $i\leq t$ by \eqref{eq:prefiltration}).
These two operations on prefiltrations can be combined: we denote by $\Gb[/t+s]$ the prefiltration given by $G_{i/t+s}=G_{i+s}/G_{t}$, it can be obtained applying first the operation $/t$ and then the operation $+s$ (hence the notation).

If $G$ is a nilpotent group then the lower central series is a filtration.
More generally, if $G_{1}\leq G_{0}$ is a normal subgroup then $G_{i+1}=[G_{i},G_{1}]$ defines a prefiltration (that has finite length if and only if $G_{1}$ is nilpotent), see for example \cite[Theorem 5.3]{MR0207802}.
If $\Gb$ is a prefiltration and $\bar d = (d_{i})_{i\in\N} \subset\N$ is a superadditive sequence (i.e.\ $d_{i+j}\geq d_{i}+d_{j}$ for all $i,j\in\N$; by convention $d_{-1}=-\infty$) then $\Gb^{\bar d}$, defined by
\begin{equation}
\label{eq:Gbard}
G^{\bar d}_{i} = G_{j} \quad \text{whenever} \quad d_{j-1} < i \leq d_{j},
\end{equation}
is again a prefiltration.

We define $\Gb$-polynomial maps by induction on the length of the prefiltration.
\begin{definition}
\label{def:polynomial}
Let $\Gamma$ be any set and $\calT$ be a set of partially defined maps $T:\Gamma \supset \dom(T) \to \Gamma$.
Let $\Gb$ be a prefiltration of length $d\in \{-\infty\}\cup\N$.
A map $g\from\Gamma\to G_{0}$ is called \emph{$\Gb$-polynomial} (with respect to $\calT$) if either $d=-\infty$ (so that $g$ identically equals the identity) or for every $T\in\calT$ there exists a $\Gb[+1]$-polynomial map $D_{T}g$ such that
\begin{equation}
D_{T}g = g\inv Tg := g\inv (g\circ T)
\quad\text{on}\quad
\dom T.
\end{equation}
We write $\poly$ for the set of $\Gb$-polynomial maps, usually suppressing any reference to the set of maps $\calT$ that will remain fixed for most of the article.
\end{definition}
Informally, a map $g:\Gamma\to G_{0}$ is polynomial if every discrete derivative $D_{T}g$ is polynomial ``of lower degree'' (the ``degree'' of a $\Gb$-polynomial map would be the length of the prefiltration $\Gb$, but we prefer not to use this notion since it is necessary to keep track of the prefiltration $\Gb$ anyway).
The connection with Leibman's notion of vector degree is provided by \eqref{eq:Gbard}: a map has vector degree $\bar d$ with respect to a prefiltration $\Gb$ if and only if it is $\Gb^{\bar d}$-polynomial.

Note that if a map $g$ is $\Gb$-polynomial, then the map $gG_{t}$ is $\Gb[/t]$-polynomial for any $t\in\N$ (but the converse is false).
We abuse the notation by saying that $g$ is $\Gb[/t]$-polynomial if $gG_{t}$ is $\Gb[/t]$-polynomial.
In assertions that hold for all $T\in\calT$ we omit the subscript in $D_{T}$.

The next theorem is the basic result about $\Gb$-polynomials.
\begin{theorem}
\label{thm:poly-group}
For every prefiltration $\Gb$ of length $d\in\{-\infty\}\cup\N$ the following holds.
\begin{enumerate}
\item\label{poly-group:commutator}
Let $t_{i}\in\N$ and $g_{i} \from \Gamma \to G$ be maps such that $g_{i}$ is $\Gb[/(d+1-t_{1-i})+t_{i}]$-polynomial for $i=0,1$.
Then the commutator $[g_{0},g_{1}]$ is $\Gb[+t_{0}+t_{1}]$-polynomial.
\item\label{poly-group:product}
Let $g_{0},g_{1} \from \Gamma \to G$ be $\Gb$-polynomial maps.
Then the product $g_{0}g_{1}$ is also $\Gb$-polynomial.
\item\label{poly-group:inverse}
Let $g \from \Gamma \to G$ be a $\Gb$-polynomial map.
Then its pointwise inverse $g\inv$ is also $\Gb$-polynomial.
\end{enumerate}
\end{theorem}
\begin{proof}
We induct on $d$.
If $d=-\infty$ then the group $G_{0}$ is trivial and the conclusion hold trivially.
Let $d\geq 0$ and assume that the conclusion holds for all smaller values of $d$.

We prove part \eqref{poly-group:commutator} using descending induction on $t=t_{0}+t_{1}$.
We clearly have $[g_{0},g_{1}]\subset G_{t}$.
If $t\geq d+1$ there is nothing left to show.
Otherwise it remains to show that $D[g_{0},g_{1}]$ is $\Gb[+t+1]$-polynomial.
To this end we use the commutator identity
\begin{multline}
\label{eq:derivative-of-commutator}
D[g_0,g_1]
=
[g_0,D g_1] \cdot [[g_0,D g_1],[g_0,g_1]] \\
\cdot [[g_0,g_1], Dg_1]
\cdot [[g_{0},g_{1}Dg_{1}],Dg_{0}] \cdot [Dg_{0},g_{1}Dg_{1}].
\end{multline}
We will show that the second to last term is $\Gb[+t+1]$-polynomial, the argument for the other terms is similar.
Note that $Dg_{0}$ is $\Gb[/(d+1-t_{1})+t_{0}+1]$-polynomial.
By inner induction hypothesis it suffices to show that $[g_{0},g_{1}Dg_{1}]$ is $\Gb[/(d-t_{0})+t_{1}]$-polynomial.
But the prefiltration $\Gb[/(d-t_{0})]$ has smaller length than $\Gb$, and by the outer induction hypothesis we can conclude that $g_{1}Dg_{1}$ is $\Gb[/(d-t_{0})+t_{1}]$-polynomial.
Moreover, $g_{0}$ is clearly $\Gb[/(d-t_{0}-t_{1})]$-polynomial, and by the outer induction hypothesis its commutator with $g_{1}Dg_{1}$ is $\Gb[/(d-t_{0})+t_{1}]$-polynomial as required.

Provided that each multiplicand in \eqref{eq:derivative-of-commutator} is $\Gb[+t+1]$-polynomial we can conclude that $D[g_{0},g_{1}]$ is $\Gb[+t+1]$-polynomial by the outer induction hypothesis.

Part \eqref{poly-group:product} follows immediately by the Leibniz rule
\begin{equation}
\label{eq:leibniz}
D(g_{0}g_{1}) = Dg_{0} [Dg_{0},g_{1}] Dg_{1}
\end{equation}
from \eqref{poly-group:commutator} with $t_{0}=1$, $t_{1}=0$ and the induction hypothesis.

To prove part \eqref{poly-group:inverse} notice that
\begin{equation}
D (g\inv)
=
g (D g)\inv g\inv
=
[g\inv,Dg] (Dg)\inv.
\end{equation}
By the induction hypothesis the map $g\inv$ is $\Gb[/d]$-polynomial, the map $Dg$ is $\Gb[+1]$-polynomial and the map $(Dg)\inv$ is $\Gb[+1]$-polynomial.
Thus also $D(g\inv)$ is $\Gb[+1]$-polynomial by \eqref{poly-group:commutator} and the induction hypothesis.
\end{proof}
Discarding some technical information that was necessary for the inductive proof we can write the above theorem succinctly as follows.
\begin{corollary}
\label{cor:poly-group}
Let $\Gb$ be a prefiltration of length $d$.
Then the set $\poly$ of $\Gb$-polynomials on $\Gamma$ is a group under pointwise operations and admits a canonical prefiltration of length $d$ given by
\[
\poly \geq \poly[{\Gb[+1]}] \geq \dots \geq \poly[{\Gb[+d+1]}].
\]
\end{corollary}
Clearly, every subgroup $F\leq\poly$ admits a canonical prefiltration $\Fb$ given by
\begin{equation}
\label{eq:F-prefiltration}
F_{i} := F \cap \poly[{\Gb[+i]}].
\end{equation}
\begin{remark}
If $\Gamma$ is a group, then we recover \cite[Proposition 3.7]{MR1910931} setting
\begin{equation}
\calT = \{T_{b}:n\mapsto nb,\dom(T_{b})=\Gamma, \text{ where } b\in\Gamma\}.
\end{equation}
\end{remark}
\begin{example}
If $\Gamma$ is a group and
\begin{equation}
\label{eq:T}
\calT = \{T_{a,b}:\Gamma\to\Gamma, n\mapsto anb,\text{ where } a,b\in\Gamma\},
\end{equation}
then every group homomorphism $g\from\Gamma\to G_{1}$ is polynomial.
In particular, every homomorphism to a nilpotent group is polynomial with respect to the lower central series.

This can be seen by induction on the length $d$ of the prefiltration $\Gb$ as follows.
If $d=-\infty$, then there is nothing to show.
Otherwise write
\begin{equation}
D_{T_{a,b}}g(n)=g(n)\inv g(anb)=[g(n),g(a)\inv] g(ab).
\end{equation}
By the induction hypothesis $gG_{d}$ is $\Gb[/d]$-polynomial and the constant maps $g(a)\inv$, $g(ab)$ are $\Gb[+1]$-polynomial since they take values in $G_{1}$.
Hence $D_{T_{a,b}}g$ is $\Gb[+1]$-polynomial by Theorem~\ref{thm:poly-group}.
\end{example}
We will encounter further concrete examples of polynomials in Proposition~\ref{prop:mono-poly} and Lemma~\ref{lem:poly-fvip}.

\subsection{IP-polynomials}
\renewcommand*{\poly}[1][\Gb]{P(\Fine,#1)}
In this article we are interested in the case $\Gamma=\Fine$, where $\Fine$ is the partial semigroup\footnote{A partial semigroup \cite{MR1262304} is a set $\Gamma$ together with a partially defined operation $*\from\Gamma\times\Gamma\to\Gamma$ that is associative in the sense that $(a*b)*c=a*(b*c)$ whenever both sides are defined.} of finite subsets of $\N$ with the operation $\alpha*\beta=\alpha\cup\beta$ that is only defined if $\alpha$ and $\beta$ are disjoint.
It is partially ordered by the relation
\[
\alpha<\beta :\iff \max\alpha < \min\beta.
\]
Note that in particular $\emptyset<\alpha$ and $\alpha<\emptyset$ for any $\alpha\in\Fine$.

The set $\calT$ is then given by
\begin{equation}
\calT = \{T_{\alpha}:\beta\mapsto \alpha*\beta,\dom(T_{\alpha})=\{\beta:\alpha\cap\beta=\emptyset\}, \text{ where } \alpha\in\Gamma\}.
\end{equation}
If $T=T_{\alpha}$ then we also write $D_{\alpha}$ instead of $D_{T_{\alpha}}$.
We write $\polyn \leq \poly$ for the subgroup of polynomials that vanish at $\emptyset$ and call its members \emph{VIP systems}.
For every $g\in\polyn$ and $\beta\in\Fine$ we have
\begin{equation}
\label{eq:polyn-g1}
g(\beta) = g(\emptyset) D_{\beta}g(\emptyset) \in G_{1}.
\end{equation}
Therefore the symmetric derivative $\sD$, defined by
\begin{equation}
\label{eq:symm-der}
\sD_{\beta}g(\alpha) := D_{\beta}g(\alpha)g(\beta)\inv = g(\alpha)\inv g(\alpha\cup\beta) g(\beta)\inv,
\end{equation}
maps $\polyn$ into $\polyn[{\Gb[+1]}]$.
Moreover, $\polyn$ admits the canonical prefiltration of length $d-1$ given by
\[
\polyn \geq \polyn[{\Gb[+1]}] \geq \dots \geq \polyn[{\Gb[+d]}].
\]
There is clearly no need to keep track of values of VIP systems at $\emptyset$, so we consider them as functions on $\Fin:=\Fine\setminus\{\emptyset\}$.

The group $\polyn$ can be alternatively characterized by $\polyn=\{1_{G}\}$ for prefiltrations $\Gb$ of length $d=-\infty,0$ and
\[
g\in\polyn \iff g:\Fin\to G_{1} \text{ and } \forall\beta\in\Fin\, \sD_{\beta}g\in\polyn[{\Gb[+1]}].
\]
This characterization shows that if $G$ is an abelian group with the standard filtration $G_{0}=G_{1}=G$, $G_{2}=\{1_{G}\}$, then $\polyn$ is just the set of IP systems in $G$.

\subsection{IP-polynomials in several variables}
As we have already mentioned in the introduction, the inductive procedure that has been so far utilized in all polynomial extensions of Szemer\'edi's theorem inherently relies on polynomials in several variables.
We find it more convenient to define polynomials in $m$ variables not on $\Fin^{m}$, but rather on the subset $\Fin^{m}_{<} \subset \Fin^{m}$ that consists of \emph{ordered} tuples, that is,
\[
\Fin^{m}_{<} = \{ (\alpha_{1},\dots,\alpha_{m}) \in \Fin^{m} : \alpha_{1}<\dots<\alpha_{m} \}.
\]
Analogously, $\Fin^{\omega}_{<}$ is the set of infinite increasing sequences in $\Fin$.
We will frequently denote elements of $\Fin^{m}_{<}$ or $\Fin^{\omega}_{<}$ by $\vec\alpha=(\alpha_{1},\alpha_{2},\dots)$.
\begin{definition}
Let $\Gb$ be a prefiltration and $F\leq\polyn$ a subgroup.
We define the set $\PE{F}{m}$ of \emph{polynomial expressions} in $m$ variables by induction on $m$ as follows.
We set $\PE{F}{0}=\{1_{G}\}$ and we let $\PE{F}{m+1}$ be the set of functions $g \from \Fin^{m+1}_{<} \to G_{0}$ such that
\[
g(\alpha_{1},\dots,\alpha_{m+1}) = W^{\alpha_{1},\dots,\alpha_{m}}(\alpha_{m+1})S(\alpha_{1},\dots,\alpha_{m}),
\]
where $S\in\PE{F}{m}$ and $W^{\alpha_{1},\dots,\alpha_{m}}\in F$ for every $\alpha_{1}<\dots<\alpha_{m}$.
\end{definition}
Note that $\PE{F}{1}=F$.
\begin{lemma}
Suppose that $F$ is invariant under conjugation by constant functions.
Then, for every $m$, the set $\PE{F}{m}$ is a group under pointwise operations and admits a canonical prefiltration given by $(\PE{F}{m})_{i}=\PE{(F_{i})}{m}$.

If $K\leq F$ is a subgroup that is invariant under conjugation by constant functions then $\PE{K}{m} \leq \PE{F}{m}$ is also a subgroup.
\end{lemma}
\begin{proof}
We induct on $m$.
For $m=0$ there is nothing to show.
Let
\[
R_{j}\in \PE{(F_{t_{j}})}{m+1}:(\alpha_{1},\dots,\alpha_{m+1})\mapsto W_{j}^{\alpha_{1},\dots,\alpha_{m}}(\alpha_{m+1})S_{j}(\alpha_{1},\dots,\alpha_{m}),
\quad j=0,1
\]
be polynomial expressions in $m+1$ variables.
Suppressing the variables $\alpha_{1},\dots,\alpha_{m}$ we have
\[
R_{0}R_{1}\inv(\alpha_{m+1})
=
W_{0}(\alpha_{m+1})
\underbrace{\left(S_{0}S_{1}\inv W_{1}\inv S_{1} S_{0}\inv\right)}_{\in F}(\alpha_{m+1})
S_{0}S_{1}\inv,
\]
so that $R_{0}R_{1}\inv \in \PE{F}{m+1}$.
Hence $\PE{F}{m+1}$ is a group.

In order to show that $\PE{(\Fb)}{m+1}$ is indeed a prefiltration we have to verify that
\begin{align*}
[R_{0},R_{1}]
&=
[W_{0}S_{0},W_{1}S_{1}]
\in (F_{t_{0}+t_{1}})^{m+1}.
\end{align*}
This follows from the identity
\begin{align*}
[xy,uv]
&=
[x,u] [x,v] [[x,v],[x,u]] [[x,u],v]\\
&\quad\cdot [[x,v] [x,u] [[x,u],v],y]\\
&\quad\cdot [y,v] [y,u] [[y,u],v].
\end{align*}
It is clear that $\PE{K}{m}\leq \PE{F}{m}$ is a subgroup provided that both sets are groups.
\end{proof}
For every $m\in\N$ there is a canonical embedding $\PE{F}{m}\leq \PE{F}{m+1}$ that forgets the last variable.
Thus we can talk about
\[
\PE{F}{\omega}:=\injlim_{m\in\N} \PE{F}{m} = \bigcup_{m\in\N} \PE{F}{m},
\]
this is a group of maps defined on $\Fin^{\omega}_{<}$ with prefiltration $(\PE{F}{\omega})_{i}=\PE{(F_{i})}{\omega}$.

\subsection{Polynomial-valued polynomials}
\begin{definition}
Let $\Gb$ be a filtration of length $d$.
A \emph{VIP group} is a subgroup $F\leq\polyn$ that is closed under conjugation by constant functions and under $\sD$ in the sense that for every $g\in F$ and $\alpha\in\Fin$ the symmetric derivative $\sD_{\alpha}g$ lies in $F_{1}$ (defined in \eqref{eq:F-prefiltration}).
\end{definition}
In particular, the group $\polyn$ itself is VIP.
\begin{proposition}
\label{prop:substitution-poly}
Let $F\leq \polyn$ be a VIP group.
Then for every $g\in \PE{F}{m}$ the substitution map
\begin{equation}
\label{eq:substitution}
h:\vec\beta=(\beta_{1},\dots,\beta_{m}) \in \Fin^{m}_{<}
\mapsto
(g[\vec\beta] : \vec\alpha\in\Fin^{\omega}_{<} \mapsto g(\cup_{i\in\beta_{1}}\alpha_{i},\dots,\cup_{i\in\beta_{m}}\alpha_{i}))
\end{equation}
lies in $\PE{\polyn[\PE{F}{\omega}]}{m}$.
\end{proposition}
\begin{proof}
We proceed by induction on $m$.
In case $m=0$ there is nothing to show, so suppose that the assertion is known for $m$ and consider $g\in\PE{F}{m+1}$.
By definition we have
\[
g(\alpha_{1},\dots,\alpha_{m+1}) = W^{\alpha_{1},\dots,\alpha_{m}}(\alpha_{m+1})S(\alpha_{1},\dots,\alpha_{m})
\]
and
\[
h(\beta_{1},\dots,\beta_{m+1})(\vec\alpha) = W^{\cup_{i\in\beta_{1}}\alpha_{i},\dots,\cup_{i\in\beta_{m}}\alpha_{i}}[\beta_{m+1}](\vec\alpha) S[\beta_{1},\dots,\beta_{m}](\vec\alpha).
\]
In view of the induction hypothesis it remains to verify that the map
\[
\tilde h : \beta\mapsto(\vec\alpha \mapsto W^{\cup_{i\in\beta_{1}}\alpha_{i},\dots,\cup_{i\in\beta_{m}}\alpha_{i}}[\beta](\vec\alpha)),
\quad \beta>\beta_{m}>\dots>\beta_{1},
\]
is in $\polyn[\PE{F}{\omega}]$.
The fact that $\tilde h(\beta)\in\PE{F}{\omega}$ for all $\beta$ follows by induction on $\abs{\beta}$ using the identity
\begin{multline*}
W^{\cup_{i\in\beta_{1}}\alpha_{i},\dots,\cup_{i\in\beta_{m}}\alpha_{i}}[\beta\cup\{b\}](\vec\alpha)
=\\
W^{\cup_{i\in\beta_{1}}\alpha_{i},\dots,\cup_{i\in\beta_{m}}\alpha_{i}}(\alpha_{b}) \sD_{\cup_{i\in\beta} \alpha_{i}}W^{\cup_{i\in\beta_{1}}\alpha_{i},\dots,\cup_{i\in\beta_{m}}\alpha_{i}}(\alpha_{b}) W^{\cup_{i\in\beta_{1}}\alpha_{i},\dots,\cup_{i\in\beta_{m}}\alpha_{i}}[\beta](\vec\alpha)
\end{multline*}
that holds whenever $b>\beta>\beta_{m}>\dots>\beta_{1}$.
In order to see that $\tilde h$ is polynomial in $\beta$ observe that
\[
\sD_{\gamma}\tilde h(\beta) : \vec\alpha \mapsto \sD_{\cup_{i\in\gamma}\alpha_{i}}W^{\cup_{i\in\beta_{1}}\alpha_{i},\dots,\cup_{i\in\beta_{m}}\alpha_{i}}(\cup_{i\in\beta}\alpha_{i}),
\quad
\beta>\gamma>\beta_{m}>\dots>\beta_{1}.
\qedhere
\]
\end{proof}

\subsection{Monomial mappings}
In this section we verify that monomial mappings into nilpotent groups in the sense of Bergelson and Leibman \cite[\textsection 1.3]{MR1972634} are polynomial in the sense of Definition~\ref{def:polynomial}.

For a sequence of finite sets $R=(R_{0},R_{1},\dots)$ only finitely many of which are non-empty and a set $\alpha$ write
\[
R[\alpha] := \alpha^{0}\times R_{0} \uplus \alpha^{1}\times R_{1} \uplus \dots
\]
Here the symbol $\uplus$ denotes disjoint union and $\alpha^{i}$ are powers of the set $\alpha$ (note that $\alpha^{0}$ consists of one element, the empty tuple).
\begin{proposition}
\label{prop:mono-poly}
Let $\Gb$ be a prefiltration of length $d$ and $N\subset \N$ any subset.
Let $\gb \from R[N] \to G$, $x\mapsto g_{x}$ be a mapping such that $\gb(N^{i}\times R_{i}) \subset G_{i}$ for every $i\in\N$ and $\prec$ be any linear ordering on $R[N]$.
Then the map
\[
g \from \Fin(N) \to G,
\quad
\alpha \mapsto \prod_{j\in R[\alpha]}^{\prec} g_{j}
\]
is $\Gb$-polynomial on the partial semigroup $\Fin(N)$ (here the symbol $\prec$ on top of $\prod$ indicates the order of factors in the product).
\end{proposition}
\begin{proof}
We induct on the length of the prefiltration $\Gb$.
If $d=-\infty$ then there is nothing to prove.
Otherwise let $\beta\in\Fin(N)$.
We have to show that $D_{\beta}g$ is $\Gb[+1]$-polynomial.

Let $B \subset R[N]$ be a finite set and $A\subset B$.
By induction on the length of an initial segment of $A$ (that proceeds by pulling the terms $g_{j}$, $j\in A$, out of the double product one by one, leaving commutators behind) we see that
\begin{equation}
\label{eq:der}
\prod_{j\in B}^{\prec} g_{j}
= \prod_{j\in A}^{\prec} g_{j}
\prod_{j\in B\setminus A}^{\prec} \prod_{k\in A^{\leq d}}^{\succ-\mathrm{lexicographic}}g_{j,k},
\end{equation}
where $A^{\leq d}$ is the set of all tuples of elements of $A$ with at most $d$ coordinates in $N$ and
\[
g_{j,\emptyset}=g_{j},
\quad
g_{j,(k_{0},\dots,k_{i})} =
\begin{cases}
[g_{j,(k_{0},\dots,k_{i-1})},g_{k_{i}}] & \text{if } j\prec k_{0} \prec \dots \prec k_{i},\\
1 & \text{otherwise}.
\end{cases}
\]

Let $\alpha\in\Fin(N)$ be disjoint from $\beta$.
Applying \eqref{eq:der} with $A:=R[\alpha]$ and $B:=R[\alpha\cup\beta]$ we obtain
\[
D_{\beta}g(\alpha) = \prod_{j\in R[\alpha\cup\beta]\setminus R[\alpha]}^{\prec} \prod_{k\in R[\alpha]^{\leq d}}^{\succ-\mathrm{lexicographic}}g_{j,k},
\]
where $g_{j,(k_{0},\dots,k_{i})} \in G_{l+l_{0}+\dots+l_{i}}$ provided that $j\in \alpha^{l}\times R_{l}$ and $k_{0}\in \alpha^{l_{0}}\times R_{l_{0}}, \dots, k_{i}\in \alpha^{l_{i}}\times R_{l_{i}}$.

The double product can be rewritten as $\prod_{l\in S[\alpha]}^{\prec'}h_{l}$ for some sequence of finite sets $S$, an ordering $\prec'$ on $S[N']$, where $N'=N\setminus\beta$, and $\hb : S[N']\to G$.
The sequence of sets $S$ is obtained by the requirement
\[
(R[\alpha\cup\beta]\setminus R[\alpha]) \times R[\alpha]^{\leq d} = S[\alpha]
\]
for every $\alpha\subset N'$.
The lexicographic ordering on $(R[N]\setminus R[N']) \times R[N']^{\leq d}$ induces an ordering $\prec'$ on $S[N']$.
Define $h_{z}=g_{j,k}$ if $(j,k)$ corresponds to $z\in S[N']$.

By construction we have $\hb((N')^{i}\times S_{i}) \subset G_{i+1}$ since each element of $R[N]\setminus R[N']$ has at least one coordinate in $N$ but not $N'$.
Thus $D_{\beta}g$ is $\Gb[+1]$-polynomial by the induction hypothesis.
\end{proof}
\begin{corollary}
\label{cor:set-monomials}
Let $G$ be a nilpotent group with lower central series
\[
G=G_{0}=G_{1}\geq\dots\geq G_{s}\geq G_{s+1}=\{1_{G}\},
\]
let $\gb : N^{d}\to G$ be an arbitrary mapping and $\prec$ be any linear ordering on $N^{d}$.
Then the map
\[
g\from\Fin(N)\to G,
\quad
\alpha \mapsto \prod_{j\in \alpha^{d}}^{\prec} g_{j}
\]
is polynomial on the partial semigroup $\Fin(N)$ with respect to the filtration
\begin{equation}
\label{eq:scalar-poly-filtration-2}
G_{0} \geq
\underbrace{G_{1} \geq \dots \geq G_{1}}_{d \text{ times}} \geq \dots \geq
\underbrace{G_{s} \geq \dots \geq G_{s}}_{d \text{ times}} \geq G_{s+1}.
\end{equation}
\end{corollary}

\section{Topological multiple recurrence}
\label{sec:hales-jewett}
In this section we refine the nilpotent Hales--Jewett theorem due to Bergelson and Leibman \cite[Theorem 0.19]{MR1972634} using the induction scheme from \cite[Theorem 3.4]{MR1715320}.
This allows us to deduce a \emph{multiparameter} nilpotent Hales--Jewett theorem that will be ultimately applied to polynomial-valued polynomials mappings.

\subsection{PET induction}
First we describe the PET (polynomial exhaustion technique) induction scheme \cite{MR912373}.
For a polynomial $g\in\polyn$ define its \emph{level} $l(g)$ as the greatest integer $l$ such that $g\in\polyn[{\Gb[+l]}]$.
We define an equivalence relation on the set of non-zero $\Gb$-polynomials by $g\sim h$ if and only if $l(g)=l(h)<l(g\inv h)$.
Transitivity and symmetry of $\sim$ follow from Theorem~\ref{thm:poly-group}.

\begin{definition}
A \emph{system} is a finite subset $A\subset\polyn$.
The \emph{weight vector} of a system $A$ is the function
\[
l\mapsto\text{the number of equivalence classes modulo }\sim\text{ of level }l\text{ in }A.
\]
The lexicographic ordering is a well-ordering on the set of weight vectors and the \emph{PET induction} is induction with respect to this ordering.
\end{definition}
\begin{proposition}
\label{prop:PET}
Let $A$ be a system, $h\in A$ be a mapping of maximal level and $B\subset G_{1}$, $M\subset\Fin$ be finite sets.
Then the weight vector of the system
\[
A'' = \{ b h\inv g \sD_{\alpha}g b\inv, \quad g\in A, \alpha\in M, b\in B \} \setminus \{1_{G}\}
\]
precedes the weight vector of $A$.
\end{proposition}
\begin{proof}
We claim first that the weight vector of the system
\[
A' = \{ h\inv g \sD_{\alpha}g, \quad \alpha\in M, g\in A \} \setminus \{1_{G}\}
\]
precedes the weight vector of $A$.
Indeed, if $l(g)<l(h)$, then $g\sim h\inv g \sD_{\alpha}g$.
If $l(g)=l(\tilde g)=l(h)$ and $g\sim \tilde g \not\sim h$, then $h\inv g \sD_{\alpha}g \sim h\inv \tilde g \sD_{\tilde\alpha}\tilde g$.
Finally, if $g\sim h$, then $l(h\inv g \sD_{\alpha}g)>l(h)$.
Thus the weight vector of $A'$ does not differ from the weight of vector of $A$ before the $l(h)$-th position and is strictly smaller at the $l(h)$-th position, as required.

We now claim that the weight vector of the system
\[
A'' = \{ b gb\inv, \quad g\in A', b\in B \}
\]
coincides with the weight vector of the system $A'$.
Indeed, this follows directly from
\[
bgb\inv = g[g,b\inv] \sim g.
\qedhere
\]
\end{proof}

\subsection{Nilpotent Hales--Jewett}
The following refined version of the nilpotent IP polynomial topological mutiple recurrence theorem due to Bergelson and Leibman \cite[Theorem 0.19]{MR1972634} does not only guarantee the existence of a ``recurrent'' point, but also allows one to choose it from a finite subset $Sx$ of any given orbit.
\begin{theorem}[Nilpotent Hales--Jewett]
\label{thm:nil-hj}
Assume that $G$ acts on the right on a compact metric space $(X,\rho)$ by homeomorphisms.
For every system $A$, every $\epsilon>0$ and every $H\in\Fin$ there exists $N\in\Fin$, $N>H$, and a finite set $S\subset G$ such that for every $x\in X$ there exist a non-empty $\alpha\subset N$ and $s\in S$ such that $\rho(xsg(\alpha),xs)<\epsilon$ for every $g\in A$.
\end{theorem}
Here we follow Bergelson and Leibman and use ``Hales--Jewett'' as a shorthand for ``IP topological multiple recurrence'', although Theorem~\ref{thm:nil-hj} does not imply the classical Hales--Jewett theorem on monochrome combinatorial lines.
The fact that Theorem~\ref{thm:nil-hj} does indeed generalize \cite[Theorem 0.19]{MR1972634} follows from Corollary~\ref{cor:set-monomials} that substitutes \cite[\textsection 1 and \textsection 2]{MR1972634}.

The reason that Theorem~\ref{thm:nil-hj} does not imply the classical Hales--Jewett theorem is that it does not apply to semigroups.
However, it is stronger than van der Waerden-type topological recurrence results, since it makes no finite generation assumptions.
We refer to \cite[\textsection 5.5]{MR1972634} and \cite[\textsection 3.3]{MR1715320} for a discussion of these issues.
It would be interesting to extend Theorem~\ref{thm:nil-hj} to nilpotent semigroups (note that nilpotency of a group can be characterized purely in terms of semigroup relations).
\begin{proof}
We use PET induction on the weight vector $w(A)$.
If $w(A)$ vanishes identically then $A$ is the empty system and there is nothing to show.
Assume that the conclusion is known for every system whose weight vector precedes $w(A)$.
Let $h\in A$ be an element of maximal level, without loss of generality we may assume $h\not\equiv 1_{G}$.
Let $k$ be such that every $k$-tuple of elements of $X$ contains a pair of elements at distance $<\epsilon/2$.

We define finite sets $H_{i}\in\Fin$, finite sets $B_{i},\tilde B_{i}\subset G$, systems $A_{i}$ whose weight vector precedes $w(A)$, positive numbers $\epsilon_{i}$, and finite sets $N_{i}\in\Fin$ by induction on $i$ as follows.
Begin with $H_{0}:=H$ and $B_{0}=\tilde B_{0}=\{1_{G}\}$.
The weight vector $w(A_{i})$ of the system
\[
A_{i} := \{b h\inv g \sD_{m}g b\inv,
\quad g\in A,m\subset N_{0}\cup\dots\cup N_{i-1},b\in B_{i}\}
\]
precedes $w(A)$ by Proposition~\ref{prop:PET}.
By uniform continuity we can choose $\epsilon_{i}$ such that
\[
\rho(x,y)< \epsilon_{i} \implies \forall \tilde b\in \tilde B_{i} \quad \rho(x \tilde b,y \tilde b)<\frac{\epsilon}{2k}.
\]
By the induction hypothesis there exists $N_{i}\in\Fin$, $N_{i}>H_{i}$, and a finite set $S_{i}\subset G$ such that
\begin{equation}
\label{eq:phj-ind}
\forall x\in X \quad \exists n_{i}\subset N_{i}, s_{i}\in S_{i} \quad \forall g\in A_{i}
\quad \rho(xs_{i}g(n_{i}),xs_{i}) < \epsilon_{i}.
\end{equation}
Finally, let $H_{i+1}:=H_{i} \cup N_{i}$ and
\[
B_{i+1} := \{ s_{i}b h(\alpha_{i})\inv,
\quad \alpha_{i}\subset N_{i},s_{i}\in S_{i},b\in B_{i}\} \subset G,
\]
\[
\tilde B_{i+1} := \{ bg(m),
\quad g\in A,m \subset N_{0}\cup\dots\cup N_{i},b\in B_{i+1}\} \subset G.
\]
This completes the inductive definition.
Now fix $x\in X$.
We define a sequence of points $y_{i}$ by descending induction on $i$.
Begin with $y_{k}:=x$.
Assume that $y_{i}$ has been chosen and choose $n_{i}\subset N_{i}$ and $s_{i}\in S_{i}$ as in \eqref{eq:phj-ind}, then set $y_{i-1}:=y_{i}s_{i}$.

Finally, let $x_{0}:=y_{0}s_{0}h(n_{0})\inv$ and $x_{i+1}:=x_{i}h(n_{i+1})\inv$.
We claim that for every $g\in A$ and any $0\leq i\leq j\leq k$ we have
\begin{equation}
\label{eq:phj-inner-ind}
\rho(x_{j}g(n_{i+1}\cup\dots\cup n_{j}),x_{i}) < \frac{\epsilon}{2k}(j-i).
\end{equation}
This can be seen by ascending induction on $j$.
Let $i$ be fixed, the claim is trivially true for $j=i$.
Assume that the claim holds for $j-1$ and let $g$ be given.
Consider
\[
b:=s_{j-1}\dots s_{0}h(n_{0})\inv \dots h(n_{j-1})\inv \in B_{j}
\quad \text{and}
\]
\[
\tilde b:=b g(n_{i+1}\cup\dots\cup n_{j-1}) \in \tilde B_{j}.
\]
By choice of $n_{j}$ and $s_{j}$ we have
\[
\rho(y_{j}s_{j}b h(n_{j})\inv g(n_{i+1}\cup\dots\cup n_{j}) g(n_{i+1}\cup\dots\cup n_{j-1})\inv b\inv, y_{j}s_{j} ) < \epsilon_{j}.
\]
By definition of $\epsilon_{j}$ this implies
\[
\rho(y_{j}s_{j}b h(n_{j})\inv g(n_{i+1}\cup\dots\cup n_{j}) g(n_{i+1}\cup\dots\cup n_{j-1})\inv b\inv \tilde b, y_{j}s_{j} \tilde b ) < \frac{\epsilon}{2k}.
\]
Plugging in the definitions we obtain
\[
\rho(x_{j}g(n_{i+1}\cup\dots\cup n_{j}), x_{j-1}g(n_{i+1}\cup\dots\cup n_{j-1}) ) < \frac{\epsilon}{2k}.
\]
The induction hypothesis then yields
\begin{multline*}
\rho(x_{j}g(n_{i+1}\cup\dots\cup n_{j}),x_{i})\\
\leq
\rho(x_{j}g(n_{i+1}\cup\dots\cup n_{j}), x_{j-1}g(n_{i+1}\cup\dots\cup n_{j-1}))
+
\rho(x_{j-1}g(n_{i+1}\cup\dots\cup n_{j-1}),x_{i})\\
<
\frac{\epsilon}{2k}
+
\frac{\epsilon}{2k}((j-1)-i)
=
\frac{\epsilon}{2k}(j-i)
\end{multline*}
as required.

Recall now that by definition of $k$ there exist $0\leq i<j\leq k$ such that $\rho(x_{i},x_{j})<\frac{\epsilon}{2}$.
By \eqref{eq:phj-inner-ind} we have
\[
\rho(x_{j}g(n_{i+1}\cup\dots\cup n_{j}),x_{j})
\leq
\rho(x_{j}g(n_{i+1}\cup\dots\cup n_{j}),x_{i})+\rho(x_{i},x_{j})
<
\epsilon
\]
for every $g\in A$.
But $x_{j}=xs$ for some
\[
s \in S := S_{k}\dots S_{0}h(\Fin(N_{0}))\inv \dots h(\Fin(N_{k}))\inv,
\]
and we obtain the conclusion with $N=N_{0}\cup\dots\cup N_{k}$ and $S$ as above.
\end{proof}
We remark that \cite[Theorem 3.4]{MR1715320} provides a slightly different set $S$ that can be recovered substituting $y_{k}:=xh(N_{k})$ and $y_{i-1}:=y_{i}s_{i}h(N_{i-1})$ in the above proof and making the corresponding adjustments to the choices of $B_{i}$, $b$ and $S$.

\subsection{Multiparameter nilpotent Hales--Jewett}
We will now prove a version of the nilpotent Hales--Jewett theorem in which the polynomial configurations may depend on multiple parameters $\alpha_{1},\dots,\alpha_{m}$.
\begin{theorem}[Multiparameter nilpotent Hales--Jewett]
\label{thm:multi-nil-hj}
Assume that $G$ acts on the right on a compact metric space $(X,\rho)$ by homeomorphisms and let $m\in\N$.
For every finite set $A\subset\PE{\polyn}{m}$, every $\epsilon>0$ and every $H\in\Fin$ there exists a finite set $N\in\Fin$, $N>H$, and a finite set $S\subset G$ such that for every $x\in X$ there exists $s\in S$ and non-empty subsets $\alpha_{1}<\dots<\alpha_{m}\subset N$ such that $\rho(xsg(\alpha_{1},\dots,\alpha_{m}),xs)<\epsilon$ for every $g\in A$.
\end{theorem}
\begin{proof}
We induct on $m$.
The base case $m=0$ is trivial.
Assume that the conclusion is known for some $m$, we prove it for $m+1$.

Let $A\subset\PE{\polyn}{m+1}$ and $H$ be given.
For convenience we write $\vec\alpha=(\alpha_{1},\dots,\alpha_{m})$ and $\alpha=\alpha_{m+1}$.
By definition each $g\in A$ can be written in the form
\[
g(\alpha_{1},\dots,\alpha_{m+1}) = g_{2}^{\vec\alpha}(\alpha) g_{1}(\vec\alpha)
\]
with $g_{1}\in\PE{\polyn}{m}$ and $g_{2}^{\vec\alpha} \in \polyn$.

We apply the induction hypothesis with the system $\{g_{1}, g\in A\}$ and $\epsilon/2$, thereby obtaining a finite set $N\in\Fin$, $N>H$, and a finite set $S\subset G$.
We write ``$\vec\alpha\subset N$'' instead of ``$\alpha_{1}<\dots<\alpha_{m}\subset N$''.

By uniform continuity there exists $\epsilon'$ such that
\[
\rho(x,y)<\epsilon' \implies
\forall s\in S, \vec\alpha\subset N,g\in A \quad
\rho(xsg_{1}(\vec\alpha),ysg_{1}(\vec\alpha))<\epsilon/2.
\]

We invoke Theorem~\ref{thm:nil-hj} with the system $\{ s g_{2}^{\vec\alpha} s\inv, s\in S, \vec\alpha\subset N, g\in A\}$ and $\epsilon'$, this gives us a finite set $N'\in\Fin$, $N'>N$, and a finite set $S'\subset G$ with the following property:
for every $x\in X$ there exist $s'\in S'$ and $\alpha\subset N'$ such that
\[
\forall s\in S, \vec\alpha\subset N, g\in A \quad
\rho(xs's g_{2}^{\vec\alpha}(\alpha)s\inv,xs') < \epsilon'.
\]
By choice of $\epsilon'$ this implies
\[
\forall s\in S, \vec\alpha\subset N, g\in A \quad
\rho(xs'sg_{2}^{\vec\alpha}(\alpha) g_{1}(\vec\alpha), xs'sg_{1}(\vec\alpha)) < \epsilon/2.
\]
By choice of $N$ and $S$, considering the point $xs'$, we can find $\vec\alpha\subset N$ and $s\in S$ such that
\[
\forall g\in A \quad \rho(xs'sg_{1}(\vec\alpha),xs's) < \epsilon/2.
\]
Combining the last two inequalities we obtain
\[
\forall g\in A \quad \rho(xs'sg(\vec\alpha,\alpha),xs's) < \epsilon.
\]
This yields the conclusion with finite sets $N\cup N'$ and $S'S$.
\end{proof}

The combinatorial version is derived using the product space construction of Furstenberg and Weiss \cite{MR531271}.
\begin{corollary}
\label{cor:color-multi-nil-hj}
Let $\Gb$ be a filtration on a countable nilpotent group $G$, $m\in\N$, $A\subset\PE{\polyn}{m}$ a finite set, and $l\in\N_{>0}$.
Then there exists $N\in\N$ and finite sets $S,T\subset G$ such that for every $l$-coloring of $T$ there exist $\alpha_{1}<\dots<\alpha_{m}\subset N$ and $s\in S$ such that the set $\{sg(\vec\alpha), g\in A\}$ is monochrome (and in particular contained in $T$).
\end{corollary}
\begin{proof}
Let $X:=l^{G}$ be the compact metrizable space of all $l$-colorings of $G$ with the right $G$-action $xg(h)=x(gh)$.
We apply Theorem~\ref{thm:multi-nil-hj} to this space, the system $A$, the set $H=\emptyset$, and an $\epsilon>0$ that is sufficiently small to ensure that $\rho(x,x')<\epsilon$ implies $x(e_{G})=x'(e_{G})$.

This yields certain $N\in\N$ and $S\subset G$ that enjoy the following property:
for every coloring $x\in X$ there exist $\alpha_{1}<\dots<\alpha_{m}\subset N$ and $s\in S$ such that $\{sg(\vec\alpha),g\in A\}$ is monochrome.
Observe that this property only involves a finite subset $T=\cup_{g\in A}Sg(\Fin(N)^{m}_{<})\subset G$.
\end{proof}
In the proof of our measurable recurrence result we will apply this combinatorial result to polynomial-valued polynomial mappings.
We encode all the required information in the next corollary.
\begin{corollary}
\label{cor:pair-color}
Let $m\in\N$, $K\leq F\leq \polyn$ be VIP groups, and $\FE\leq\PE{\polyn}{\omega}$ be a countable subgroup that is closed under substitutions $g\mapsto g[\vec\beta]$ (recall \eqref{eq:substitution}).

Then for any finite subsets $(R_{i})_{i=0}^{t}\subset\PE{K}{m}\cap\FE$ and $(W_{k})_{k=0}^{v-1}\subset\PE{F}{m}\cap\FE$ there exist $N,w\in\N$ and $(L_{i},M_{i})_{i=1}^{w}\subset(\PE{K}{N}\cap\FE)\times(\PE{F}{N}\cap\FE)$ such that for every $l$-coloring of the latter set there exists an index $a$ and sets $\beta_{1}<\dots<\beta_{m}\subset N$ such that the set $(L_{a}R_{i}[\vec\beta],M_{a}W_{k}[\vec\beta]L_{a}\inv)_{i,k}$ is monochrome (and in particular contained in the set $(L_{i},M_{i})_{i=1}^{w}$).
We may assume $L_{1}\equiv 1_{G}$.
\end{corollary}
\begin{proof}
By Proposition~\ref{prop:substitution-poly} the maps $\vec\beta \mapsto (R_{i}[\vec\beta],W_{k}[\vec\beta]R_{i}[\vec\beta])$ are polynomial expressions with values in $\PE{K}{\omega}\times\PE{F}{\omega}$.
By the assumption they also take values in $\FE\times\FE$.
Given an $l$-coloring $\chi$ of $(\PE{K}{\omega}\cap\FE)\times(\PE{F}{\omega}\cap\FE)$ we pass to the $l$-coloring $\tilde\chi(g,h)=\chi(g,hg\inv)$.
Corollary~\ref{cor:color-multi-nil-hj} then provides the desired $N$ and $(L_{i},M_{i})_{i=1}^{w}=T\cup S$.
\end{proof}

\section{FVIP groups}
\label{sec:fvip}
In this section we consider a class of IP-polynomials that IP-converge to orthogonal projections.
\begin{definition}
\label{def:fvip}
An \emph{FVIP group} is a finitely generated VIP group.
An \emph{FVIP system} is a member of some FVIP group.
\end{definition}
The main result about FVIP groups is the following nilpotent version of \cite[Theorem 1.8]{MR1417769} and \cite[Theorem 1.9]{MR2246589} that will be used to construct ``primitive extensions'' (we will recall the definitions of a primitive extension and an IP-limit in due time).
\begin{theorem}
\label{thm:fvip-proj}
Let $\Gb$ be a prefiltration of finite length and $F\leq\polyn$ be an FVIP group.
Suppose that $G_{0}$ acts on a Hilbert space $H$ by unitary operators and that for each $(g_{\alpha})_{\alpha}\in F$ the weak limit $P_{g}=\wIPlim_{\alpha\in\Fin} g_{\alpha}$ exists.
Then
\begin{enumerate}
\item each $P_{g}$ is an orthogonal projection and
\item these projections commute pairwise.
\end{enumerate}
\end{theorem}
The finite generation assumption cannot be omitted in view of a counterexample in \cite{MR1417769}.
We begin with some algebraic preliminaries.

\subsection{Hirsch length}
We use Hirsch length of a group as a substitute for the concept of the rank of a free $\Z$-module.
Recall that a subnormal series in a group is called \emph{polycyclic} if the quotients of consecutive subgroups in this series are cyclic and a group is called \emph{polycyclic} if it admits a polycyclic series.
\begin{definition}
The \emph{Hirsch length} $h(G)$ of a polycyclic group $G$ is the number of infinite quotients of consecutive subgroups in a polycyclic series of $G$.
\end{definition}
Recall that the Hirsch length is well-defined by the Schreier refinement theorem, see e.g.\ \cite[Theorem 5.11]{MR1307623}.
For a finitely generated nilpotent group $G$ with a filtration $\Gb$ one has
\[
h(G) = \sum_{i} \rank G_{i}/G_{i+1}.
\]
\begin{lemma}
\label{lem:fin-ind-hirsch}
Let $G$ be a finitely generated nilpotent group.
Then for every subgroup $V\leq G$ we have that $h(V)=h(G)$ if and only if $[G:V]<\infty$.
\end{lemma}
\begin{proof}
If $[G:V]<\infty$, then we can find a finite index subgroup $W\leq V$ that is normal in $G$, and the equality $h(G)=h(W)=h(V)$ follows from the Schreier refinement theorem.

Let now $\Gb$ be the lower central series of $G$.
Let $V\leq G$ be a subgroup with $h(V)=h(G)$ and assume in addition that $G_{i}\leq V$ for some $i=1,\dots,d+1$.
We show that $[G:V]<\infty$ by induction on $i$.
For $i=1$ the claim is trivial and for $i=d+1$ it provides the desired equivalence.

Assume that the claim holds for some $i$.
Let $V_{i} := V\cap G_{i}$ be the filtration on $V$ induced by $\Gb$ and assume $V_{i+1}=G_{i+1}$.
By the assumption we have
\[
\sum_{j=1}^{d} \rank G_{j}/G_{j+1} = h(G) = h(V) = \sum_{j=1}^{d} \rank V_{j}/V_{j+1},
\]
and, since $V_{j}/V_{j+1}\cong V_{j}G_{j+1}/G_{j+1} \leq G_{j}/G_{j+1}$ for every $j$, this implies that $V_{i}/G_{i+1} \leq G_{i}/G_{i+1}$ is a finite index subgroup.
Let $K\subset G_{i}$ be a finite set such that $KV_{i}/G_{i+1}=G_{i}/G_{i+1}$.
Then $KV\leq G$ is a subgroup and a finite index extension of $V$.
Moreover, we have $KV\supseteq G_{i}$, and by the first part of the lemma we obtain $h(KV)=h(V)$.

By the induction hypothesis $KV$ has finite index in $G$, so the index of $V$ is also finite.
\end{proof}

\begin{lemma}
\label{lem:fin-ext}
Let $G$ be a finitely generated nilpotent group with a filtration $\Gb$ of length $d$ and let $V\leq G$ a subgroup.
Then for every $j=1,\dots,d+1$ and every $g\in G$ there exist at most finitely many finite index extensions of $V$ of the form $\<V,gc\>$ with $c\in G_{j}$.
\end{lemma}
\begin{proof}
We use descending induction on $j$.
The case $j=d+1$ is clear, so assume that the conclusion is known for $j+1$ and consider some $g\in G$.

Let $c_{a}$, $a=0,1$ be elements of $G_{j}$ such that $\<V,gc_{a}\>$ are finite index extensions of $V$.
Then also $\<VG_{j+1},gc_{a}\>/G_{j+1}$ is a finite index extension of $VG_{j+1}/G_{j+1}$, so that there exists an $m>0$ such that $(gc_{a}G_{j+1})^{m}\in VG_{j+1}/G_{j+1}$ for $a=0,1$.

Since the elements $c_{a}G_{j+1}$ are central in $G/G_{j+1}$ this implies $(c_{0}\inv c_{1})^{m}G_{j+1}\in (VG_{j+1}\cap G_{j})/G_{j+1}$.
But the latter group is a subgroup of the finitely generated abelian group $G_{j}/G_{j+1}$, so that $c_{0}\inv c_{1} \in K(VG_{j+1}\cap G_{j})$ for some finite set $K\subset G_{j}$ that does not depend on $c_{0},c_{1}$.

Multiplying $c_{1}$ with an element of $V$ we may assume that $c_{1} \in c_{0}KG_{j+1}$.
By the induction hypothesis for each $g'\in gc_{0}K$ there exist at most finitely many finite index extensions of the form $\<V,g'c'\>$ with $c'\in G_{j+1}$, so we have only finitely many extensions of the form $\<V,gc_{1}\>$ as required.
\end{proof}
\begin{corollary}
\label{cor:fin-ext}
Let $G$ be a finitely generated nilpotent group and $V$ be a subgroup.
Then there exist at most finitely many finite index extensions of $V$ of the form $\<V,c\>$.
\end{corollary}
\begin{proof}
Consider any filtration $\Gb$ and apply Lemma~\ref{lem:fin-ext} with $j=1$ and $g=1_{G}$.
\end{proof}
The following example shows that Corollary~\ref{cor:fin-ext} may fail for virtually nilpotent groups.
Consider the semidirect product $G=\Z_{2} \ltimes \Z$ that is associated to the inversion action $\pi:\Z_{2} \curvearrowleft \Z$ given by $\pi(\bar a)(b)=(-1)^{a}b$.
Then $G_{2}=[G,G]=2\Z$ is an abelian subgroup of index $4$ and $G_{i+1}=[G,G_{i}]=2^{i}\Z$ for all $i\in\N$, so $G$ is not nilpotent.
Let $V = \{0\} \leq G$ be the trivial subgroup.
Since we have $(\bar{1}a)^{2}=0 \in V$ for any $a\in\Z$, each group of the form $\<V,\bar{1}a\>$ is an extension of $V$ with index $2$.
On the other hand, for every value of $a$ we obtain a different extension.

\subsection{Partition theorems for IP-rings}
An \emph{IP-ring} is a subset of $\Fin$ that consists of all finite unions of a given strictly increasing chain $\alpha_{0}<\alpha_{1}<\dots$ of elements of $\Fin$ \cite[Definition 1.1]{MR833409}.
In particular, $\Fin$ is itself an IP-ring (associated to the chain $\{0\}<\{1\}<\dots$).
Polynomials are generally assumed to be defined on $\Fin$ even if we manipulate them only on some sub-IP-ring of $\Fin$.

Since we will be dealing a lot with assertions about sub-IP-rings we find it convenient to introduce a shorthand notation.
If some statement holds for a certain sub-IP-ring $\Fin'\subset\Fin$ then we say that it holds \emph{without loss of generality} (wlog).
In this case we reuse the symbol $\Fin$ to denote the sub-IP-ring on which the statement holds (in particular this IP-ring may change from use to use).
This is the only sense in which the phrase ``wlog'' will be used in \textsection\ref{sec:fvip} and \textsection\ref{sec:primitive}.
With this convention the basic Ramsey-type theorem about IP-rings reads as follows.
\begin{theorem}[Hindman \cite{MR0349574}]
\label{thm:hindman}
Every finite coloring of $\Fin$ is wlog monochrome.
\end{theorem}
Since ``wlog'' is an existential quantifier, it is important where it appears in a sentence.
For instance, Theorem~\ref{thm:hindman} is not the same as the assertion ``wlog every finite coloring of $\Fin$ is monochrome'', since the latter would mean that there exists a sub-IP-ring on which \emph{every} coloring is monochrome.

As a consequence of Hindman's theorem~\ref{thm:hindman}, a map from $\Fin$ to a compact metric space for every $\epsilon>0$ wlog has values in an $\epsilon$-ball.
As the next lemma shows, for polynomial maps into compact metric groups the ball can actually be chosen to be centered at the identity.
In a metric group we denote the distance to the identity by $\dint{\cdot}$.
\begin{lemma}
\label{lem:near-identity}
Let $\Gb$ be a prefiltration in the category of compact metric groups and $P \in \polyn$.
Then for every $\epsilon>0$ we have wlog $\dint{P} < \epsilon$.
\end{lemma}
\begin{proof}
We induct on the length of the prefiltration $\Gb$.
If the prefiltration is trivial, then there is nothing to show, so assume that the conclusion is known for $\Gb[+1]$.

Let $\delta,\delta'>0$ be chosen later.
By compactness and Hindman's theorem~\ref{thm:hindman} we may wlog assume that the image $P(\Fin)$ is contained in some ball $B(g,\delta)$ with radius $\delta$ in $G_{1}$.
By uniform continuity of the group operation we have $\sD_{\beta}P(\alpha) \in B(g\inv,\delta')$ for any $\alpha>\beta\in\Fin$ provided that $\delta$ is small enough depending on $\delta'$.
On the other hand, for a fixed $\beta$, by the induction hypothesis we have wlog $\dint{\sD_{\beta}P} < \delta'$, so that $\dint{g\inv}<2\delta'$.
By continuity of inversion this implies $\dint{g}<\epsilon/2$ provided that $\delta'$ is small enough.
This implies $\dint{P}<\epsilon$ provided that $\delta$ is small enough.
\end{proof}

\begin{corollary}[{\cite[Proposition 1.1]{MR2246589}}]
\label{cor:poly-subg}
Let $\Wb$ be a prefiltration, $A\subset \polyn[\Wb]$ be finite, and $V\leq W$ be a finite index subgroup.
Then wlog for every $g\in A$ we have $g(\Fin) \subset V$.
\end{corollary}
\begin{proof}
Let $g\in A$.
Passing to a subgroup we may assume that $V$ is normal.
Taking the quotient by $V$, we may assume that $W$ is finite and $V=\{1_{W}\}$.
By Lemma~\ref{lem:near-identity} with an arbitrary discrete metric we may wlog assume that $g\equiv 1_{W}$.
\end{proof}

In course of proof of Theorem~\ref{thm:fvip-proj} it will be more convenient to use a convention for the symmetric derivative that differs from \eqref{eq:symm-der}, namely
\[
\rD_{\alpha}g(\beta) = g(\alpha)\inv D_{\alpha}g(\beta).
\]
Clearly a VIP group is also closed under $\rD$.
\begin{lemma}
\label{lem:deri-poly}
Let $F$ be a VIP group, $W\leq F$ be a subgroup and $V\leq W$ be a finite index subgroup.
Suppose that $g\in F$ is such that the symmetric derivative $\rD_{\alpha}g\in W$ for all $\alpha$.
Then wlog every symmetric derivative $\rD_{\alpha}g$, $\alpha\in\Fin$, coincides with an element of $V$ on some sub-IP-ring of the form $\{\beta\in\Fin : \beta>\beta_{0}\}$.
\end{lemma}
\begin{proof}
Since $V$ has finite index and by Hindman's theorem~\ref{thm:hindman} we can wlog assume that $\rD_{\alpha}g\in w\inv V$ for some $w\in W$ and all $\alpha$.
Assume that $w\not\in V$.
Let
\[
h(\alpha) := w \rD_{\alpha}g =
\begin{cases}
w, & \alpha=\emptyset\\
v_{\alpha}\in V & \text{otherwise}.
\end{cases}
\]
Let $\alpha_{1}<\dots<\alpha_{d}$ be non-empty, by induction on $d$ we see that $D_{\alpha_{d}}\dots D_{\alpha_{1}}h(\alpha) \in V$ for all $\alpha\neq\emptyset$ and $D_{\alpha_{d}}\dots D_{\alpha_{1}}h(\emptyset) \in V w^{(-1)^{d}}V$.

On the other hand the map $\alpha\mapsto h(\alpha)(\beta)$ is $\Gb[+1]$-polynomial on $\{\alpha : \alpha\cap\beta=\emptyset\}$ for fixed $\beta$.
Therefore $D_{\alpha_{d}}\dots D_{\alpha_{1}}h(\emptyset)$ vanishes at all $\beta > \alpha_{d}$, that is, $w$ coincides with an element of $V$ on $\{\beta : \beta>\alpha_{d}\}$.
\end{proof}
It is possible to see Lemma~\ref{lem:deri-poly} (and Lemma~\ref{lem:notinK} later on) as a special case of Corollary~\ref{cor:poly-subg} by considering the quotient of $\polyn$ by the equivalence relation of equality on IP-rings of the form $\{\alpha:\alpha>\alpha_{0}\}$, but we prefer not to set up additional machinery.

In order to apply the above results we need a tool that provides us with finite index subgroups.
To this end recall the following multiparameter version of Hindman's theorem~\ref{thm:hindman}.
\begin{theorem}[Milliken \cite{MR0373906}, Taylor \cite{MR0424571}]
\label{thm:milliken-taylor}
Every finite coloring of $\Fin_{<}^{k}$ is wlog monochrome.
\end{theorem}

The next lemma is a substitute for \cite[Lemma 1.6]{MR1417769} in the non-commutative case.
This is the place where the concept of Hirsch length is utilized.
\begin{lemma}
\label{lem:sfi}
Let $G$ be a finitely generated nilpotent group and $g \from\Fin\to G$ be any map.
Then wlog there exist a natural number $l>0$ and a subgroup $W\leq G$ such that for any $\alpha_{1}<\dots<\alpha_{l} \in \Fin$ the elements $g_{\alpha_{1}},\dots,g_{\alpha_{l}}$ generate a finite index subgroup of $W$.
\end{lemma}
\begin{proof}
By the Milliken--Taylor theorem~\ref{thm:milliken-taylor} we may wlog assume that for each $l\leq h(G)+1$ the Hirsch length $h(\<g_{\alpha_{1}},\dots,g_{\alpha_{l}}\>)$ does not depend on $(\alpha_{1},\dots,\alpha_{l})\in\Fin^{l}_{<}$.
Call this value $h_{l}$.
It is an increasing function of $l$ that is bounded by $h(G)$, hence there exists an $l$ such that $h_{l}=h_{l+1}$.
Fix some $(\alpha_{1},\dots,\alpha_{l})\in\Fin^{l}_{<}$ and let $V:=\<g_{\alpha_{1}},\dots,g_{\alpha_{l}}\>$.

Since $h_{l+1}=h_{l}$ and by Lemma~\ref{lem:fin-ind-hirsch}, we see that $\<V,g_{\alpha}\>$ is a finite index extension of $V$ for each $\alpha > \alpha_{l}$.
By Corollary~\ref{cor:fin-ext} and Hindman's Theorem~\ref{thm:hindman} we may wlog assume that each $g_{\alpha}$ lies in one such extension $W$.
By definition of $h_{l}$ this implies that wlog for every $(\alpha_{1},\dots,\alpha_{l})\in\Fin^{l}_{<}$ the Hirsch length of the group $\<g_{\alpha_{1}},\dots,g_{\alpha_{l}}\> \leq W$ is $h(W)$.
Hence each $\<g_{\alpha_{1}},\dots,g_{\alpha_{l}}\> \leq W$ is a finite index subgroup by Lemma~\ref{lem:fin-ind-hirsch}.
\end{proof}

\subsection{IP-limits}
Let $X$ be a topological space, $m\in\N$ and $g:\Fin_{<}^{m}\to X$ be a map.
We call $x\in X$ an IP-limit of $g$, in symbols $\IPlim_{\vec\alpha}g_{\vec\alpha}=x$, if for every neighborhood $U$ of $x$ there exists $\alpha_{0}$ such that for all $\vec\alpha\in\Fin_{<}^{m}$, $\vec\alpha>\alpha_{0}$, one has $g_{\vec\alpha}\in U$.

By the Milliken--Taylor theorem~\ref{thm:milliken-taylor} and a diagonal argument, cf.\ \cite[Lemma 1.4]{MR833409}, we may wlog assume the existence of an IP-limit (even of countably many IP-limits) if $X$ is a compact metric space, see \cite[Theorem 1.5]{MR833409}.

If $X$ is a Hilbert space with the weak topology we write $\wIPlim$ instead of $\IPlim$ to stress the topology.

Following a tradition, we write arguments of maps defined on $\Fin$ as subscripts in this section.
We also use the notation and assumptions of Theorem~\ref{thm:fvip-proj}.

The next lemma follows from the equivalence of the weak and the strong topology on the unit sphere of $H$ and is stated for convenience.
\begin{lemma}
\label{lem:norm}
Assume that $f\in\fix P_{g}$, that is, that $\wIPlim_{\alpha}g_{\alpha}f=f$.
Then also $\IPlim_{\alpha}g_{\alpha}f=f$ (in norm).
\end{lemma}
For any subgroup $V\leq F$ we write $P_{V}$ for the orthogonal projection onto the space $\bigcap_{g\in V} \fix P_{g}$.
\begin{lemma}
\label{lem:gp}
Assume that $V=\<g_{1},\dots,g_{s}\>$ is a finitely generated group and that $P_{g_{1}},\dots,P_{g_{s}}$ are commuting projections.
Then $P_{V}=\prod_{i=1}^{s} P_{g_{i}}$.
\end{lemma}
\begin{proof}
Clearly we have $P_{V} \leq \prod_{i=1}^{s} P_{g_{i}}$, so we only need to prove that each $f$ that is fixed by $P_{g_{1}},\dots,P_{g_{s}}$ is also fixed by $P_{g}$ for any other $g\in V$.

To this end it suffices to show that if $f$ is fixed by $P_{g}$ and $P_{h}$ for some $g,h\in V$, then it is also fixed by $P_{gh\inv}$.
Lemma~\ref{lem:norm} shows that $\IPlim_{\alpha} g_{\alpha}f=f$ and $\IPlim_{\alpha} h_{\alpha}f=f$.
Since each $h_{\alpha}$ is unitary we obtain $\IPlim_{\alpha} h_{\alpha}\inv f=f$.
Since each $g_{\alpha}$ is isometric, this implies
\[
\wIPlim g_{\alpha}h_{\alpha}\inv f = \IPlim_{\alpha} g_{\alpha}h_{\alpha}\inv f = f
\]
as required.
\end{proof}

The next lemma is the main tool to ensure IP-convergence to zero.
\begin{lemma}[{\cite[Lemma 1.7]{MR1417769}}]
\label{lem:comm-proj}
Let $(P_{\alpha})_{\alpha\in\Fin}$ be a family of commuting orthogonal projections on a Hilbert space $H$ and $f\in H$.
Suppose that, whenever $\alpha_{1}<\dots<\alpha_{l}$, one has $\prod_{i=1}^{l}P_{\alpha_{i}}f=0$.
Then $\IPlim_{\alpha} \norm{P_{\alpha}f} = 0$.
\end{lemma}

Finally, we also need a van der Corput-type estimate.
\begin{lemma}[{\cite[Lemma 5.3]{MR833409}}]
\label{lem:vdC}
Let $(x_{\alpha})_{\alpha\in\Fin}$ be a bounded family in a Hilbert space $H$.
Suppose that
\[
\IPlim_{\beta} \IPlim_{\alpha} \<x_{\alpha},x_{\alpha\cup\beta}\>=0.
\]
Then wlog we have
\[
\wIPlim_{\alpha} x_{\alpha} = 0.
\]
\end{lemma}

\begin{proof}[Proof of Theorem~\ref{thm:fvip-proj}]
We proceed by induction on the length of the prefiltration $\Gb$.
If $\Gb$ is trivial there is nothing to prove.
Assume that the conclusion is known for $\Gb[+1]$.

First, we prove that $P_{g}$ is an orthogonal projection for any $g\in F$ (that we now fix).
Since $P_{g}$ is clearly contractive it suffices to show that it is a projection.

By Lemma~\ref{lem:sfi} we may assume that, for some $l>0$ and any $\alpha_{1}<\dots<\alpha_{l}$, the derivatives $\rD_{\alpha_{1}}g,\dots,\rD_{\alpha_{l}}g$ generate a finite index subgroup of some $W\leq F_{1}$ (recall that $F_{1}=F\cap \polyn[{\Gb[+1]}]$).
We split
\begin{equation}
\label{eq:splitting}
H = \bigcap_{V\leq W}\ker P_{V} \oplus \overline{\lin}\Big(\bigcup_{V\leq W}\im P_{V}\Big) =: H_{0} \cup H_{1},
\end{equation}
where $V$ runs over the finite index subgroups of $W$.
It suffices to show $P_{g}f=P_{g}^{2}f$ for each $f$ in one of these subspaces.

\paragraph{Case 0}
Let $f\in H_{0}$ and $\alpha_{1}<\dots<\alpha_{l}$.
By choice of $W$ we know that
\[
V:=\<\rD_{\alpha_{1}}g,\dots,\rD_{\alpha_{l}}g\> \leq W
\]
is a finite index subgroup.
Since the projections $P_{\rD_{\alpha_{i}}g}$ commute by the inductive hypothesis, their product equals $P_{V}$ (Lemma~\ref{lem:gp}), and we have $P_{V}f=0$ by the assumption.

By Lemma~\ref{lem:comm-proj} this implies $\IPlim_{\alpha} \norm{ P_{\rD_{\alpha}g} f }=0$.
Therefore
\[
\IPlim_{\alpha} \abs[\Big]{ \IPlim_{\beta} \< (\rD_{\alpha}g)_{\beta}f, g_{\alpha}\inv f \> }
\leq
\IPlim_{\alpha} \norm{ \wIPlim_{\beta} (\rD_{\alpha}g)_{\beta} f }
=
\IPlim_{\alpha} \norm{ P_{\rD_{\alpha}g} f }
= 0,
\]
so that
\[
\IPlim_{\alpha} \IPlim_{\beta} \< g_{\alpha\cup\beta}f, g_{\beta}f \> = 0.
\]
By Lemma~\ref{lem:vdC} this implies $P_{g}f=0$ (initially only wlog, but we have assumed that the limit exists on the original IP-ring).

\paragraph{Case 1}
Let $V\leq W$ and $f=P_{V}f$, by linearity we may assume $\norm{f}=1$.
Let $\rho$ be a metric for the weak topology on the unit ball of $H$ with $\rho(x,y)\leq\norm{x-y}$.
Let $\epsilon>0$.
By definition of IP-convergence and by uniform continuity of $P_{g}$ there exists $\alpha_{0}$ such that
\[
\forall\alpha>\alpha_{0} \quad
\rho(g_{\alpha}f,P_{g}f) < \epsilon
\text{ and }
\rho(P_{g}g_{\alpha}f,P_{g}^{2}f) < \epsilon.
\]
By Lemma~\ref{lem:deri-poly} we can choose $\alpha>\alpha_{0}$ such that $\rD_{\alpha}g$ coincides with an element of $V$ on some sub-IP-ring, so that in particular $P_{\rD_{\alpha}g}f=f$.
By Lemma~\ref{lem:norm} there exists $\beta_{0}>\alpha$ such that
\[
\forall\beta>\beta_{0}
\quad
\norm{ (\rD_{\alpha}g)_{\beta}f-f } < \epsilon.
\]
Applying $g_{\beta}g_{\alpha}$ to the difference on the left-hand side we obtain
\[
\norm{ g_{\alpha\cup\beta}f - g_{\beta}g_{\alpha}f } < \epsilon,
\text{ so that }
\rho(g_{\alpha\cup\beta}f, g_{\beta}g_{\alpha}f) < \epsilon.
\]
Observe that $\alpha\cup\beta > \alpha_{0}$, so that
\[
\rho(P_{g} f, g_{\beta}g_{\alpha}f) < 2\epsilon.
\]
Taking IP-limit along $\beta$ we obtain
\[
\rho(P_{g} f, P_{g}g_{\alpha}f) \leq 2\epsilon.
\]
A further application of the triangle inequality gives
\[
\rho(P_{g} f, P_{g}^{2}f) < 3\epsilon,
\]
and, since $\epsilon>0$ was arbitrary, we obtain $P_{g}f=P_{g}^{2}f$.

\paragraph{Commutativity of projections}
Let us now prove the second conclusion, namely that $P_{g}$ and $P_{g'}$ commute for any $g,g'\in F$.
Observe that the function $\alpha\mapsto g_{\alpha}$ can be seen as a polynomial-valued function in $\polyn[{\poly}]$ whose values are constant polynomials.
Moreover we can consider the constant function in $\poly[{\Fb}]$ whose value is $g'$.
Taking their commutator we see that
\[
\alpha\mapsto [g_{\alpha},g']
\quad \in \polyn[{\poly}],
\]
and, since $F$ is a VIP group, this map in fact lies in $\polyn[{\Fb}]$.
By \eqref{eq:polyn-g1} it takes values in $F_{1}$.
By Lemma~\ref{lem:sfi} we may assume that for any $\alpha_{1}<\dots<\alpha_{l}$ the maps $[g_{\alpha_{1}},g'],\dots,[g_{\alpha_{l}},g']$ generate a finite index subgroup of some $W\leq F_{1}$.
Interchanging $g$ and $g'$ and repeating this argument we may also wlog assume that for any $\alpha_{1}<\dots<\alpha_{l'}$ the maps $[g'_{\alpha_{1}},g],\dots,[g'_{\alpha_{l'}},g]$ generate a finite index subgroup of some $W'\leq F_{1}$.
Consider the splitting
\begin{equation}
\label{eq:splitting2}
H = \Big(\bigcap_{V\leq W}\ker P_{V} \cap \bigcap_{V'\leq W'}\ker P_{V'}\Big)
\oplus \overline{\lin}\Big(\bigcup_{V\leq W}\im P_{V} \cup \bigcup_{V'\leq W'}\im P_{V'}\Big)
=: H_{0} \cup H_{1}.
\end{equation}

\paragraph{Case 0}
Let $f\in H_{0}$.
As above we have $\IPlim_{\alpha} \norm{P_{[g_{\alpha},g']}f} = 0$, and in particular
\begin{multline*}
0
= \IPlim_{\alpha} \< \wIPlim_{\beta} [g_{\alpha},g'_{\beta}]f, g_{\alpha}\inv P_{g'}f\>\\
= \IPlim_{\alpha} \IPlim_{\beta} \<g_{\alpha}g'_{\beta}f, g'_{\beta} P_{g'}f\>
= \IPlim_{\alpha} \<g_{\alpha} P_{g'}f, P_{g'}f\>,
\end{multline*}
since $\IPlim_{\beta}g'_{\beta} P_{g'}f = P_{g'}f$ by Lemma~\ref{lem:norm}.
Hence $P_{g}P_{g'}f \perp P_{g'}f$, which implies $P_{g}P_{g'}f=0$ since $P_{g}$ is an orthogonal projection.

Interchanging the roles of $g$ and $g'$, we also obtain $P_{g'}P_{g}f=0$.

\paragraph{Case 1}
Let $V\leq W$ and $f=P_{V}f$.
By Corollary~\ref{cor:poly-subg} we may wlog assume that $[g_{\alpha},g']\in V$ for all $\alpha$.
Let $\alpha$ be arbitrary, by Lemma~\ref{lem:norm} the limit
\[
\IPlim_{\beta} [g_{\alpha},g'_{\beta}] f = f
\]
also exists in norm.
Therefore
\[
g_{\alpha} P_{g'} f
= \wIPlim_{\beta} g_{\alpha} g'_{\beta} f
= \wIPlim_{\beta} g'_{\beta} g_{\alpha} [g_{\alpha},g'_{\beta}] f
= \wIPlim_{\beta} g'_{\beta} g_{\alpha} f
= P_{g'} g_{\alpha} f.
\]
Taking IP-limits on both sides we obtain
\[
P_{g} P_{g'} f = P_{g'} P_{g} f.
\]
The case $V'\leq W'$ and $f=P_{V'}f$ can be handled in the same way.
\end{proof}
If the group $G$ acts by measure-preserving transformations then the Hilbert space projections identified in Theorem~\ref{thm:fvip-proj} are in fact conditional expectations as the following folklore lemma shows.
\begin{lemma}
Let $X$ be a probability space and $(T_{\alpha})_{\alpha}$ be a net of operators on $L^{2}(X)$ induced by measure-preserving transformations.
Assume that $T_{\alpha}\to P$ weakly for some projection $P$.
Then $P$ is a conditional expectation.
\end{lemma}
\begin{proof}
Note that $\im P \cap L^{\infty}(X)$ is dense in $\im P$.

Let $f,g\in\im P \cap L^{\infty}(X)$.
Since the weak and the norm topology coincide on the unit sphere of $L^{2}(X)$, we have $\norm{T_{\alpha}f-f}_{2} \to 0$ and $\norm{T_{\alpha}g-g}_{2} \to 0$.
Therefore
\begin{multline*}
\norm{P(fg)-fg}_{2}
\leq
\limsup_{\alpha} \norm{T_{\alpha}(fg)-fg}_{2}
=
\limsup_{\alpha} \norm{(T_{\alpha}f-f)T_{\alpha}g+f(T_{\alpha}g-g)}_{2}\\
\leq
\limsup_{\alpha} \norm{T_{\alpha}f-f}_{2} \norm{T_{\alpha}g}_{\infty}+\norm{f}_{\infty} \norm{T_{\alpha}g-g}_{2}
=
0.
\end{multline*}
This shows that $\im P\cap L^{\infty}(X)$ is an algebra, and the assertion follows.
\end{proof}

\subsection{Generalized polynomials and examples of FVIP groups}
In order to obtain some tangible combinatorial applications of our results we will need non-trivial examples of FVIP groups.
The first example somewhat parallels Proposition~\ref{prop:mono-poly}.
\begin{lemma}[\cite{MR2246589}]
\label{lem:fvip-monomial}
Let $(n^{1}_{i})_{i\in\N},\dots,(n^{a}_{i})_{i\in\N} \subset \Z$ be any sequences, $(G,+)$ be a commutative group, $(y_{i})_{i\in\N}\subset G$ be any sequence, and $d\in\N$.
Then the maps of the form
\begin{equation}
\label{eq:fvip-monomial}
v(\alpha) = \sum_{i_{1}<\dots<i_{e}\in\alpha} n^{j_{1}}_{i_{1}}\cdots n^{j_{e-1}}_{i_{e-1}} y_{i_{e}},
\quad e\leq d, 1\leq j_{1},\dots,j_{e-1} \leq a,
\end{equation}
generate an FVIP subgroup $F\leq\polyn$, where the prefiltration $\Gb$ is given by $G_{0}=\dots=G_{d}=G$, $G_{d+1}=\{1_{G}\}$.
\end{lemma}
Maps of the form \eqref{eq:fvip-monomial} were originally studied in connection with admissible generalized polynomials (Definition~\ref{def:generalized-poly}).
We will not return to them in the sequel and a proof of the above lemma is included for completeness.
\begin{proof}
The group $F$ is by definition finitely generated and closed under conjugation by constants since $G$ is commutative.
It remains to check that the maps of the form \eqref{eq:fvip-monomial} are polynomial and that the group $F$ is closed under symmetric derivatives.

To this end we induct on $d$.
The cases $d=0,1$ are clear (in the latter case the maps \eqref{eq:fvip-monomial} are IP-systems), so let $d>1$ and consider a map $v$ as in \eqref{eq:fvip-monomial} with $e=d$.
For $\beta<\alpha$ we have
\begin{multline*}
\sD_{\beta} v(\alpha)
=
-v(\alpha)+v(\beta\cup\alpha)-v(\beta)\\
=
\sum_{k=1}^{d-1} \sum_{i_{1}<\dots<i_{k}\in\beta, i_{k+1}<\dots<i_{d}\in\alpha} n^{j_{1}}_{i_{1}}\cdots n^{j_{e-1}}_{i_{d-1}} y_{i_{d}}\\
=
\sum_{k=1}^{d-1} \sum_{i_{1}<\dots<i_{k}\in\beta} n^{j_{1}}_{i_{1}}\cdots n^{j_{k}}_{i_{k}} \underline{\sum_{i_{k+1}<\dots<i_{d}\in\alpha} n^{j_{k+1}}_{i_{k+1}}\cdots n^{j_{e-1}}_{i_{d-1}} y_{i_{d}}}.
\end{multline*}
The underlined expression is $\Gb[+1]$-polynomial by the induction hypothesis and lies in $F$ by definition.
Since this holds for every $\beta$, the map $v$ is $\Gb$-polynomial.
Since the derivatives are in $F$ for every map $v$, the group $F$ is FVIP.
\end{proof}

The following basic property of FVIP groups will be used repeatedly.
\begin{lemma}
\label{lem:fvip-plus-fvip}
Let $F,F' \leq\polyn$ be FVIP groups.
Then the group $F\vee F'$ is also FVIP.
\end{lemma}
\begin{proof}
The group $F\vee F'$ is clearly finitely generated and invariant under conjugation by constants.
Closedness under $\sD$ follows from the identity
\begin{equation}
\label{eq:sD-Leibniz}
\sD_{m}(gh) = h\inv \sD_{m}g g(m) h \sD_{m}h g(m)\inv.
\qedhere
\end{equation}
\end{proof}

We will now elaborate on the example that motivated Bergelson, H\aa{}land Knutson and McCutcheon to study FVIP systems in the first place \cite{MR2246589}.
They have shown that ranges of generalized polynomials from a certain class necessarily contain FVIP systems.

We begin by recalling the definition of the appropriate class.
We denote the integer part function by $\floor{\cdot}$, the nearest integer function by $\nint{\cdot} = \floor{\cdot+1/2}$ and the distance to nearest integer by $\dint{a} = \abs{a-\nint{a}}$ (this is consistent with the notation for general metric groups applied to $\R/\Z$).
\begin{definition}
\label{def:generalized-poly}
The set $\calG$ of \emph{generalized polynomials} (in $l$ variables) is the smallest $\Z$-algebra of functions $\Z^{l}\to\Z$ that contains $\Z[x_{1},\dots,x_{l}]$ such that for every $p_{1},\dots,p_{t}\in\calG$ and $c_{1},\dots,c_{t}\in\R$ the map $\floor{\sum_{i=1}^{t}c_{i}p_{i}}$ is in $\calG$.
The notion of degree is extended from $\Z[x_{1},\dots,x_{l}]$ to $\calG$ inductively by requiring $\deg p_{0}p_{1}\leq \deg p_{0}+\deg p_{1}$, $\deg (p_{0}+p_{1}) \leq \max(\deg p_{0},\deg p_{1})$, and $\deg \floor{\sum_{i=1}^{t}c_{i}p_{i}} \leq \max_{i}\deg p_{i}$, the degree of each generalized polynomial being the largest number with these properties.

The set of $\calG_{a}$ of \emph{admissible} generalized polynomials is the smallest ideal of $\calG$ that contains the maps $x_{1},\dots,x_{l}$ and is such that for every $p_{1},\dots,p_{t}\in\calG_{a}$, $c_{1},\dots,c_{t}\in\R$, and $0<k<1$ the map $\floor{\sum_{i=1}^{t}c_{i}p_{i}+k}$ is in $\calG_{a}$.
\end{definition}
The construction of FVIP systems in the range of an admissible generalized polynomial in \cite{MR2246589} proceeds by induction on the polynomial and utilizes Lemma~\ref{lem:fvip-monomial} at the end.
We give a softer argument that gives a weaker result in the sense that it does not necessarily yield an FVIP system of the form \eqref{eq:fvip-monomial}, but requires less computation.

For a ring $R$ (with not necessarily commutative multiplication, although we will only consider $R=\Z$ and $R=\R$ in the sequel) and $d\in\N$ we denote by $R_{\bullet}^{d}$ the prefiltration (with respect to the additive group structure) given by $R_{0}=\dots=R_{d}=R$ and $R_{d+1}=\{0_{R}\}$.
\begin{lemma}
\label{lem:fvip-ring-product}
Let $F_{i} \leq \polyn[R_{\bullet}^{d_{i}}]$, $i=0,1$, be FVIP groups.
Then the pointwise products of maps from $F_{0}$ and $F_{1}$ generate an FVIP subgroup of $\polyn[R_{\bullet}^{d_{0}+d_{1}}]$.
\end{lemma}
\begin{proof}
This follows by induction on $d_{0}+d_{1}$ using the identity
\[
\sD_{\beta} vw = (v+\sD_{\beta}v+v(\beta)) \sD_{\beta} w + (\sD_{\beta} v+v(\beta))w + (\sD_{\beta}v+v)w(\beta)
\]
for the symmetric derivative of a pointwise product.
\end{proof}

Applying Lemma~\ref{lem:near-identity} to $\R/\Z$ we obtain the following.
\begin{corollary}
\label{cor:near-integer}
Let $P$ be an FVIP system in $\R$.
Then for every $\epsilon>0$ wlog $\dint{P}<\epsilon$.
\end{corollary}
This allows us to show that we can obtain $\Z$-valued FVIP systems from $\R$-valued FVIP-systems by rounding.
\begin{lemma}
\label{lem:nearest-integer-FVIP}
Let $P \in \polyn[\R_{\bullet}^{d}]$ be an FVIP system.
Then wlog $\nint{P} \in \polyn[\Z_{\bullet}^{d}]$ and $\nint{P}$ is an FVIP system.
\end{lemma}
\begin{proof}
We induct on $d$.
For $d=0$ there is nothing to show, so assume that $d>0$.
By the assumption every symmetric derivative of $(P_{\alpha})$ lies in an FVIP group of polynomials of degree $<d$ that is generated by $q_{1},\dots,q_{a}$, say.
By the induction hypothesis we know that wlog each $\nint{q_{i}}$ is again an FVIP system and by Lemma~\ref{lem:fvip-plus-fvip} they lie in some FVIP group $F$.
By Corollary~\ref{cor:near-integer} we may assume wlog that $\dint{P} < 1/12$.
Let now $\beta$ be given, by the hypothesis we have
\[
\sD_{\beta} P(\alpha) = \sum_{i}c_{i}q_{i}(\alpha)
\quad \text{for } \alpha>\beta
\]
with some $c_{i}\in \Z$.
By Corollary~\ref{cor:near-integer} we may wlog assume that $\abs{c_{i}}\cdot \dint{q_{i}}(\alpha) < 1/(4\cdot 2^{i})$ for all $\alpha>\beta$.
This implies
\[
\abs{\sD_{\beta} \nint{P}(\alpha) - \sum_{i}c_{i}\nint{q_{i}}(\alpha)} < 1/2
\quad \text{for } \alpha>\beta,
\]
so that
\[
\sD_{\beta} \nint{P}(\alpha) = \sum_{i}c_{i}\nint{q_{i}}(\alpha)
\quad \text{for } \alpha>\beta,
\]
since both sides are integer-values functions.
In fact we can do this for all $\beta$ with fixed $\max\beta$ simultaneously.
By a diagonal argument, cf.\ \cite[Lemma 1.4]{MR833409}, we may then assume that for every $\beta$ we have
\[
\sD_{\beta} \nint{P}(\alpha) = \sum_{i}c_{i}\nint{q_{i}}(\alpha)
\quad \text{for } \alpha>\beta
\]
with some $c_{i}\in\Z$.
Hence $F \vee \<\nint{P}\> \leq \polyn[\Z_{\bullet}^{d}]$ is an FVIP group.
\end{proof}
Recall that an \emph{IP-system} in $\Z^{l}$ is a family $(n_{\alpha})_{\alpha\in\Fin} \subset \Z^{l}$ such that $n_{\alpha\cup\beta}=n_{\alpha}+n_{\beta}$ whenever $\alpha,\beta\in\Fin$ are disjoint.
\begin{theorem}[{\cite[Theorem 2.8]{MR2246589}}]
\label{thm:gen-poly-FVIP}
For every generalized polynomial $p:\Z^{l}\to\Z$ and every FVIP system $(n_{\alpha})_{\alpha}$ in $\Z^{l}$ of degree at most $d$ there exists $n\in\Z$ such that the IP-sequence $(p(n_{\alpha})-n)_{\alpha\in\Fin}$ is wlog FVIP of degree at most $d\deg p$.
If $p$ is admissible, then we may assume $n=0$.
\end{theorem}
\begin{proof}
We begin with the first part.
The class of maps that satisfy the conclusion is closed under $\Z$-linear combinations by Lemma~\ref{lem:fvip-plus-fvip} and under multiplication by Lemma~\ref{lem:fvip-ring-product}.
This class clearly contains the polynomials $1,x_{1},\dots,x_{l}$.
Thus it remains to show that, whenever $p_{1},\dots,p_{t}\in\calG$ satisfy the conclusion and $c_{1},\dots,c_{t}\in\R$, the map $\floor{P}$ with $P=\sum_{i=1}^{t}c_{i}p_{i}$ also satisfies the conclusion.

By the assumption we have wlog that $(P(n_{\alpha})-C)_{\alpha}$ is an $\R$-valued FVIP system for some $C\in\R$.
By Hindman's theorem~\ref{thm:hindman} we may wlog assume that $\floor{P(n_{\alpha})}=\nint{P(n_{\alpha})-C}+n$ for some integer $n$ with $\abs{n-C}<2$ and all $\alpha\in\Fin$.
The conclusion follows from Lemma~\ref{lem:nearest-integer-FVIP}.

Now we consider admissible generalized polynomials $p$.
The conclusion clearly holds for $x_{1},\dots,x_{l}$, passes to linear combinations and passes to products with arbitrary generalized polynomials by Lemma~\ref{lem:fvip-ring-product} and the first part of the statement.
Assume now that $p_{1},\dots,p_{t}\in\calG_{a}$ satisfy the conclusion and $c_{1},\dots,c_{t}\in\R$, $0<k<1$.
Then $(P(n_{\alpha}))_{\alpha}$ with $P:=\sum_{i=1}^{t}c_{i}p_{i}$ is wlog an $\R$-valued FVIP system, and by Corollary~\ref{cor:near-integer} we have wlog $\dint{P} < \min(k,1-k)$.
This implies $\floor{P(n_{\alpha})+k}=\nint{P(n_{\alpha})}$ and this is wlog an FVIP system by Lemma~\ref{lem:nearest-integer-FVIP}.
\end{proof}
As an aside, consider the set of real-valued generalized polynomials $\mathcal{RG}$ \cite[Definition 3.1]{MR2747062} that is defined similarly to $\calG$, except that it is required to be an $\R$-algebra.
Following the proof of Theorem~\ref{thm:gen-poly-FVIP} we see that for every $p\in\mathcal{RG}$ and every FVIP system $(n_{\alpha})_{\alpha}\subset\Z^{l}$ wlog there exists a constant $C\in\R$ such that $(p(n_{\alpha})-C)_{\alpha}$ is an FVIP system.
Clearly, if $p$ is of the form $\floor{q}$ then $C\in\Z$ and if $p\in\R[x_{1},\dots,x_{l}]$ with zero constant term then $C=0$.
This, together with Corollary~\ref{cor:near-integer}, implies (an FVIP* version of) \cite[Theorem D]{MR2318563}.

Our main example (that also leads to Theorem~\ref{thm:SZ-intro}) are maps induced by admissible generalized polynomial sequences in finitely generated nilpotent groups.
\begin{lemma}
\label{lem:poly-fvip}
Let $G$ be a finitely generated nilpotent group with a filtration $\Gb$.
Let $p:\Z^{l}\to\Z$ be an admissible generalized polynomial, $(n_{\alpha})_{\alpha} \subset \Z^{l}$ be an FVIP system of degree at most $d$ and $g\in G_{d\deg p}$.
Then wlog $(g^{p(n_{\alpha})})_{\alpha}$ is an element of $\polyn$ and an FVIP system.
\end{lemma}
\begin{proof}
By Theorem~\ref{thm:gen-poly-FVIP} we can wlog assume that $(p(n_{\alpha}))_{\alpha}$ is a $\Z$-valued FVIP system of degree $\leq d\deg p$.
Using the (family of) homomorphism(s) $\Z_{\bullet}^{d\deg p}\to\Gb$, $1\mapsto g$, we see that $(g^{p(n_{\alpha})})_{\alpha}$ is contained in a finitely generated subgroup $F_{0}\leq\polyn$ that is closed under $\sD$.

Let $A\subset F_{0}$ and $B\subset G$ be finite generating sets.
Then the group generated by $bab\inv$, $a\in A$, $b\in B$, is FVIP in view of the identity \eqref{eq:sD-Leibniz}.
\end{proof}

\section{Measurable multiple recurrence}
\label{sec:primitive}
Following the general scheme of Furstenberg's proof, we will obtain our multiple recurrence theorem by (in general transfinite) induction on a suitable chain of factors of the given measure-preserving system.
For the whole section we fix a nilpotent group $G$ with a filtration $\Gb$ and an FVIP group $F\leq\polyn$.

In the induction step we pass from a factor to a ``primitive extension'' that enjoys a dichotomy: each element of $F$ acts on it either relatively compactly or relatively mixingly.
Since the reasoning largely parallels the commutative case here, we are able to refer to the article of Bergelson and McCutcheon \cite{MR1692634} for many proofs.
The parts of the argument that do require substantial changes are given in full detail.

Whenever we talk about measure spaces $(X,\calA,\mu)$, $(Y,\calB,\nu)$, or $(Z,\calC,\gamma)$ we suppose that they are regular and that $G$ acts on them on the right by measure-preserving transformations.
This induces a left action on the corresponding $L^{2}$ spaces.
Recall that to every factor map $(Z,\calC,\gamma)\to (Y,\calB,\nu)$ there is associated an essentially unique measure disintegration
\[
\gamma = \int_{y\in Y} \gamma_{y} \dif\nu(y),
\]
see \cite[\textsection 5.4]{MR603625}.
We write $\norm{\cdot}_{y}$ for the norm on $L^{2}(Z,\gamma_{y})$.
Recall also that the fiber product $Z\times_{Y}Z$ is the space $Z\times Z$ with the measure $\int_{y\in Y} \gamma_{y} \otimes \gamma_{y} \dif\nu(y)$.

\subsection{Compact extensions}
We begin with the appropriate notion of relative compactness.
Heuristically, an extension is relatively compact if it is generated by the image of a relatively Hilbert--Schmidt operator.
\begin{definition}[{\cite[Definition 3.4]{MR1692634}}]
Let $Z\to Y$ be a factor.
A \emph{$Z|Y$-kernel} is a function $H \in L^{\infty}(Z \times_{Y} Z)$ such that
\[
\int H(z_{1},z_{2}) \dif\gamma_{z_{2}}(z_{1}) = 0
\]
for a.e. $z_{2}\in Z$.
If $H$ is a $Z|Y$-kernel and $\phi\in L^{2}(Z)$ then
\[
H*\phi(z_{1}) := \int H(z_{1},z_{2}) \phi(z_{2}) \dif\gamma_{z_{1}}(z_{2}).
\]
\end{definition}
The map $\phi\mapsto H*\phi$ is a Hilbert--Schmidt operator on almost every fiber over $Y$ with uniformly bounded Hilbert--Schmidt norm.
These operators are self-adjoint provided that $H(z_{1},z_{2})=\overline{H(z_{2},z_{1})}$ a.e.

\begin{definition}[{\cite[Definition 3.6]{MR1692634}}]
Suppose that $X \to Z \to Y$ is a chain of factors, $K\leq F$ is a subgroup and $H$ is a non-trivial self-adjoint $X|Y$-kernel that is $K$-invariant in the sense that
\[
\IPlim_{\alpha} g(\alpha)H = H
\]
for every $g\in K$.
The extension $Z\to Y$ is called \emph{$K$-compact} if it is generated by functions of the form $H*\phi$, $\phi\in L^{2}(X)$.
\end{definition}

\begin{lemma}[{\cite[Remark 3.7(i)]{MR1692634}}]
Let $X \to Z \to Y$ be a chain of factors in which $Z\to Y$ is a $K$-compact extension generated by a $X|Y$-kernel $H$.
Then $H$ is in fact a $Z|Y$-kernel and $Z$ is generated by functions of the form $H*\phi$, $\phi\in L^{2}(Z)$.
\end{lemma}
\begin{proof}
Call the projection maps $\pi:X\to Z$, $\theta:X\to Y$.
Let $\phi\in L^{2}(X)$.
Since $H*\phi$ is $Z$-measurable we have
\begin{align*}
H*\phi(x)
&= \int H*\phi(x_{1}) \dif\mu_{\pi(x)}(x_{1})\\
&= \int \int H(x_{1},x_{2})\phi(w_{2}) \dif\mu_{\theta(x_{1})}(x_{2}) \dif\mu_{\pi(x)}(x_{1})\\
&= \int \int H(x_{1},x_{2})\phi(x_{2}) \dif\mu_{\theta(x)}(x_{2}) \dif\mu_{\pi(x)}(x_{1})\\
&= \int \int H(x_{1},x_{2}) \dif\mu_{\pi(x)}(x_{1}) \phi(x_{2}) \dif\mu_{\theta(x)}(x_{2})\\
&= \E(H|Z\times_{Y} X)*\phi(w).
\end{align*}
Since this holds for all $\phi$ we obtain $H=\E(H | Z\times_{Y} X)$.
Since $H$ is self-adjoint this implies that $H$ is $Z\times_{Y}Z$-measurable.
This in turn implies that $H*\phi = H*\E(\phi | Z)$ for all $\phi\in L^{2}(X)$.
\end{proof}
In view of this lemma the reference to the ambient space $X$ is not necessary in the definition of a $K$-compact extension.
Just like in the commutative case, compactness is preserved upon taking fiber products (this is only used in the part of the proof of Theorem~\ref{thm:wm} that we do not write out).
\begin{lemma}[{\cite[Remark 3.7(ii)]{MR1692634}}]
Let $Z\to Y$ be a $K$-compact extension.
Then $Z\times_{Y}Z \to Y$ is also a $K$-compact extension.
\end{lemma}

\subsection{Mixing and primitive extensions}
Now we define what we mean by relative mixing and the dichotomy between relative compactness and relative mixing.
\begin{definition}[{\cite[Definition 3.5]{MR1692634}}]
Let $Z\to Y$ be an extension.
A map $g\in F$ is called \emph{mixing on $Z$ relatively to $Y$} if for every $H\in L^{2}(Z\times_{Y} Z)$ with $\E(H | Y)=0$ one has $\wIPlim_{\alpha}g(\alpha) H = 0$.
An extension $Z\to Y$ is called \emph{$K$-primitive} if it is $K$-compact and each $g\in F\setminus K$ is mixing on $Z$ relative to $Y$.
\end{definition}

The above notion of mixing might be more appropriately called ``mild mixing'', but we choose a shorter name since there will be no danger of confusion.

The next lemma is used in the suppressed part of the proof of Theorem~\ref{thm:wm}.
\begin{lemma}[{\cite[Proposition 3.8]{MR1692634}}]
\label{lem:fiber-prod-prim}
Let $Z\to Y$ be a $K$-primitive extension.
Then $Z\times_{Y}Z \to Y$ is also a $K$-primitive extension.
\end{lemma}

Like in the commutative setting \cite[Lemma 2.8]{MR2151599} the compact part of a primitive extension is wlog closed under taking derivatives, but there is also a new aspect, namely that it is also closed under conjugation by constants.
\begin{lemma}
\label{lem:K-conj-inv}
Let $Z\to Y$ be a $K$-primitive extension.
Then $K$ is closed under conjugation by constant functions.
Moreover wlog $K$ is an FVIP group.
\end{lemma}
\begin{proof}
Let $g\in F\setminus K$, $h\in G$ and $H\in L^{2}(Z\times_{Y}Z)$ be such that $\E(H | Y)=0$.
Then
\[
\wIPlim_{\alpha} (h\inv gh)(\alpha) H
=
h\inv \wIPlim_{\alpha} g(\alpha) (hH)
=
0,
\]
so that $F\setminus K$ is closed under conjugation by constant functions, so that $K$ is also closed under conjugation by constant functions.

Since $F$ is Noetherian, the subgroup $K$ is finitely generated as a semigroup.
Fix a finite set of generators for $K$.
By Hindman's Theorem~\ref{thm:hindman} we may wlog assume that for every generator $g$ we have either $\sD_{\alpha}g\in K$ for all $\alpha\in\Fin$ or $\sD_{\alpha}g \not\in K$ for all $\alpha\in\Fin$.
In the latter case we obtain
\[
0 = \wIPlim_{\alpha,\beta} \sD_{\beta}g(\alpha) H
= \IPlim_{\alpha,\beta} g(\alpha)\inv g(\alpha\cup\beta) g(\beta)\inv H
= H,
\]
a contradiction.
Thus we may assume that all derivatives of the generators lie in $K$.
This extends to the whole group $K$ by \eqref{eq:sD-Leibniz} and invariance of $K$ under conjugation by constants.
\end{proof}

\subsection{Existence of primitive extensions}
Since our proof proceeds by induction over primitive extensions we need to know that such extensions exist.
First, we need a tool to locate non-trivial kernels.
\begin{lemma}[{\cite[Lemma 3.12]{MR1692634}}]
Let $X\to Y$ be an extension.
Suppose that $0\neq H \in L^{2}(X\times_{Y}X)$ satisfies $\E(H | Y)=0$ and that there exists $g\in F$ such that $\IPlim_{\alpha}g(\alpha)H=H$.

Then there exists a non-trivial self-adjoint non-negative definite $X|Y$-kernel $H'$ such that $\IPlim_{\alpha}g(\alpha)H'=H'$.
\end{lemma}
Second, we have to make sure that we cannot accidentally trivialize them.
\begin{lemma}[{\cite[Lemma 3.14]{MR1692634}}]
Let $Z\to Y$ be a $K$-compact extension.
Suppose that for some $g\in K$ and self-adjoint non-negative definite $Z|Y$-kernel $H$ we have
\[
\IPlim_{\alpha} \int (g(\alpha)H)(f' \otimes \bar f') \dif\tilde\gamma = 0
\]
for all $f'\in L^{\infty}(Z)$.
Then $H=0$.
\end{lemma}

The next theorem that provides existence of primitive extensions can be proved in the same way as in the commutative case \cite[Theorem 3.15]{MR1692634}.
The only change is that Theorem~\ref{thm:fvip-proj} is used instead of \cite[Theorem 2.17]{MR1692634} (note that $F$ is Noetherian, since it is a finitely generated nilpotent group).
\begin{theorem}
\label{thm:primitive-extension}
Let $X\to Y$ be a proper factor.
Then there exists a subgroup $K\leq F$ and a factor $X\to Z\to Y$ such that the extension $Z\to Y$ is proper and wlog $K$-primitive.
\end{theorem}

\subsection{Almost periodic functions}
For the rest of Section~\ref{sec:primitive} we fix a good group $\FE\leq\PE{\polyn}{\omega}$.
We will describe what we mean by ``good'' in Definition~\ref{def:good}, for the moment it suffices to say that $\FE$ is countable.
\begin{definition}[{\cite[Definition 3.1]{MR1692634}}]
Suppose that $(Z,\calC,\gamma)\to (Y,\calB,\nu)$ is a factor and $K\leq F$ a subgroup.
A function $f\in L^{2}(Z)$ is called \emph{$K$-almost periodic} if for every $\epsilon>0$ there exist $g_{1},\dots,g_{l}\in L^{2}(Z)$ and $D\in\calB$ with $\nu(D)<\epsilon$ such that for every $\delta>0$ and $R\in\PE{K}{\omega}\cap\FE$ there exists $\alpha_{0}$ such that for every $\alpha_{0}<\vec\alpha\in \Fin^{\omega}_{<}$ there exists a set $E=E(\vec\alpha)\in\calB$ with $\nu(E)<\delta$ such that for all $y\in Y\setminus(D\cup E)$ there exists $1\leq j\leq l$ such that
\[
\norm{ R(\alpha)f - g_{j} }_{y} < \epsilon.
\]
The set of $K$-almost periodic functions is denoted by $\AP(Z,Y,K)$.
\end{definition}
The next lemma says that a characteristic function that can be approximated by almost periodic functions can be replaced by an almost periodic function right away.
\begin{lemma}[{\cite[Theorem 3.3]{MR1692634}}]
\label{lem:indicator-ap}
Let $A\in\calC$ be such that $1_{A} \in\overline{\AP(Z,Y,K)}$ and $\delta>0$.
Then there exists a set $A'\subset A$ such that $\gamma(A\setminus A')<\delta$ and $1_{A'}\in \AP(Z,Y,K)$.
\end{lemma}
In the following lemma we have to restrict ourselves to $\PE{K}{\omega}\cap\FE$ since $\PE{K}{\omega}$ need not be countable.
\begin{lemma}[{\cite[Proposition 3.9]{MR1692634}}]
Let $X\to Y$ be an extension, $K\leq F$ a subgroup and $H$ a $X|Y$-kernel that is $K$-invariant.
Then wlog for all $R\in\PE{K}{\omega}\cap\FE$ and $\epsilon>0$ there exists $\alpha_{0}$ such that for all $\alpha_{0}<\vec\alpha$ we have
\[
\norm{ R(\vec\alpha)H - H } < \epsilon.
\]
\end{lemma}

With help of the above lemma we can show that in fact wlog every characteristic function can be approximated by almost periodic functions.
In view of Lemma~\ref{lem:indicator-ap} this allows us to reduce the question of multiple recurrence in a primitive extension to multiple recurrence for (relatively) almost periodic functions.
\begin{lemma}[{\cite[Theorem 3.11]{MR1692634}}]
\label{lem:ap-dense}
Let $Z\to Y$ be a $K$-compact extension.
Then wlog $\AP(Z,Y,K)$ is dense in $L^{2}(Z)$.
\end{lemma}

\subsection{Multiple mixing}
Under sufficiently strong relative mixing assumptions the limit behavior of a multicorrelation sequence $\prod_{i}S_{i}(\alpha)f_{i}$ on a primitive extension only depends on the expectations of the functions on the base space.
The appropriate conditions on the set $\{S_{i}\}_{i}$ are as follows.
\begin{definition}
\label{def:K-mixing}
Let $K\leq F$ be a subgroup.
A subset $A \subset F$ is called \emph{$K$-mixing} if $1_{G}\in A$ and $g\inv h\in F\setminus K$ whenever $g\neq h\in A$.
\end{definition}
The requirement $1_{G}\in A$ is not essential, but it is convenient in inductive arguments.
In order to apply PET induction we will need the next lemma.

We say that a subgroup $K\leq F$ is \emph{invariant under equality of tails} if whenever $S\in K$ and $T\in F$ are such that there exists $\beta\in\Fin$ with $S_{\alpha}=T_{\alpha}$ for all $\alpha>\beta$ we have $T\in K$.
Every group $K\leq F$ that is the compact part of some primitive extension has this property.
\begin{lemma}
\label{lem:notinK}
Let $K\leq F$ be a subgroup that is invariant under equality of tails.
Let $S,T\in F$ be such that $S\inv T \not\in K$.
Then wlog
\[
(S \sD_{\beta}S)\inv (T\sD_{\beta}T) \not\in K
\quad\text{and}\quad
S\inv (T\sD_{\beta} T) \not\in K
\]
for all $\beta\in\Fin$.
\end{lemma}
\begin{proof}
If the first conclusion fails then by Hindman's theorem~\ref{thm:hindman} wlog
\[
h(\alpha) := (S \sD_{\alpha}S)\inv (T\sD_{\alpha}T) \in K \text{ for all }\alpha\in\Fin
\]
and $h(\emptyset)\not\in K$.
Proceed as in the proof of Lemma~\ref{lem:deri-poly}.
Analogously for the second conclusion.
\end{proof}

The next lemma is a manifestation of the principle that compact orbits can be thought of as being constant.
\begin{lemma}[{\cite[Proposition 4.2]{MR1692634}}]
\label{lem:wm-comp}
Let $Z\to Y$ be a $K$-primitive extension, $R^{\beta}\in K$ for each $\beta\in\Fin$ and $W\in F\setminus K$.
Let also $f,f'\in L^{\infty}(Z)$ be such that either $\E(f | Y)=0$ or $\E(f' | Y)=0$.
Then wlog
\[
\IPlim_{\beta,\alpha} \norm{\E(R^{\beta}(\alpha)f W(\beta)f' | Y) } = 0.
\]
\end{lemma}

We come to the central result on multiple mixing.
\begin{theorem}[{cf.\ \cite[Theorem 4.10]{MR1692634}}]
\label{thm:wm}
Let $K\leq F$ be a subgroup.
For every $K$-mixing set $\{S_{0}\equiv 1_{G},S_{1},\dots,S_{t}\}\subset F$ the following statements hold.
\begin{enumerate}
\item For every $K$-primitive extension $Z\to Y$ and any $f_{0},\dots,f_{t}\in L^{\infty}(Z)$ we have wlog
\[
\wIPlim_{\alpha} \prod_{i=1}^{t} S_{i}(\alpha)f_{i} - \prod_{i=1}^{t} S_{i}(\alpha)\E(f_{i} | Y) = 0.
\]
\item
For every $K$-primitive extension $Z\to Y$ and any $f_{0},\dots,f_{t}\in L^{\infty}(Z)$ we have wlog
\[
\IPlim_{\alpha} \norm[\Big]{ \E \big(\prod_{i=0}^{t} S_{i}(\alpha)f_{i} \big| Y \big) - \prod_{i=0}^{t} S_{i}(\alpha)\E(f_{i} | Y) } = 0.
\]
\item
For every $K$-primitive extension $Z\to Y$, any $U_{i,j}\in K$, and any $f_{i,j}\in L^{\infty}(Z)$ we have wlog
\[
\IPlim_{\alpha} \norm[\Big]{ \E \big( \prod_{i=0}^{t} S_{i}(\alpha) \big( \prod_{j=0}^{s}U_{i,j}(\alpha)f_{i,j} \big) \big| Y \big)- \prod_{i=0}^{t} S_{i}(\alpha) \E \big( \prod_{j=0}^{s}U_{i,j}(\alpha)f_{i,j} \big| Y \big) } = 0.
\]
\end{enumerate}
\end{theorem}
We point out that the main induction loop is on the mixing set.
It is essential that, given $K\leq F$, all statements are proved simultaneously for all $K$-compact extensions since the step from weak convergence to strong convergence involves a fiber product via Lemma~\ref{lem:fiber-prod-prim}.
\begin{proof}
The proof is by PET-induction on the mixing set.
We only prove that the last statement for mixing sets with lower weight vector implies the first, the proofs of other implications are the same as in the commutative case.

By the telescope identity it suffices to consider the case $\E(f_{i_{0}} | Y)=0$ for some $i_{0}$.
By the van der Corput Lemma~\ref{lem:vdC} it suffices to show that wlog
\[
\IPlim_{\beta,\alpha} \int_{Z} \prod_{i=1}^{t} S_{i}(\alpha)f_{i} \prod_{i=1}^{t} S_{i}(\alpha\cup\beta)\bar f_{i} = 0.
\]
This limit can be written as
\[
\IPlim_{\beta,\alpha} \int_{Z} \prod_{i=1}^{t} S_{i}(\alpha)f_{i} \prod_{i=1}^{t} \underbrace{S_{i}(\alpha) \sD_{\beta}S_{i}(\alpha)}_{=:T_{i,\beta}(\alpha)} (S_{i}(\beta)\bar f_{i}).
\]
By Lemma~\ref{lem:notinK} we may wlog assume that $T_{i,\beta}\inv T_{j,\beta}$ and $S_{i}\inv T_{j,\beta}$ are mixing for all $\beta\in\Fin$ provided that $i\neq j$.
Re-indexing if necessary and using Hindman's theorem~\ref{thm:hindman} we may wlog assume $S_{i}\inv T_{i,\beta} \in K$ for all $\beta\in\Fin$ and $i\leq w$ and $S_{i}\inv T_{i,\beta} \not\in K$ for all $\beta\in\Fin$ and $i>w$ for some $w=0,\dots,t$.
Thus
\begin{equation}
\label{eq:larger-mixing-set}
S_{0},S_{1},\dots,S_{t},T_{w+1,\beta},\dots,T_{t,\beta}
\end{equation}
is a $K$-mixing set for every $\beta\neq\emptyset$.
Moreover it has the same weight vector as $\{S_{1},\dots,S_{t}\}$ since $T_{i,\beta} \sim S_{i}$.
Assume that $S_{j}$, $j\neq 0$, has the maximal level in \eqref{eq:larger-mixing-set}.
We have to show
\begin{multline*}
\IPlim_{\beta,\alpha} \int_{Z} \prod_{i=1}^{w} S_{j}\inv(\alpha)S_{i}(\alpha)(f_{i} \sD_{\beta}S_{i}(\alpha) (S_{i}(\beta)\bar f_{i}))\\
\cdot \prod_{i=w+1}^{t} S_{j}\inv(\alpha)S_{i}(\alpha)f_{i} S_{j}\inv(\alpha)T_{i,\beta}(\alpha) (S_{i}(\beta)\bar f_{i}) = 0.
\end{multline*}
For each fixed $\beta\in\Fin$ the limit along $\alpha$ comes from the $K$-mixing set
\[
S_{j}\inv S_{1},\dots,S_{j}\inv S_{t},S_{j}\inv T_{w+1,\beta},\dots,S_{j}\inv T_{t,\beta}
\]
that has lower weight vector.
Hence we can apply the induction hypothesis, thereby obtaining that the limit equals
\begin{multline*}
\IPlim_{\beta,\alpha} \int_{Z} \prod_{i=1}^{w} S_{j}\inv(\alpha)S_{i}(\alpha)\E(f_{i} \sD_{\beta}S_{i}(\alpha) (S_{i}(\beta)\bar f_{i}) | Y)\\
\cdot\prod_{i=w+1}^{t} S_{j}\inv(\alpha)S_{i}(\alpha)\E(f_{i} | Y) S_{j}\inv(\alpha)T_{i,\beta}(\alpha) \E(S_{i}(\beta)\bar f_{i} | Y)
\end{multline*}
This clearly vanishes if $i_{0}>w$, while in the case $i_{0}\leq w$ this vanishes by Lemma~\ref{lem:wm-comp}.
\end{proof}

\subsection{Multiparameter multiple mixing}
In fact we need some information about relative polynomial mixing in several variables.
First we need to say what we understand under a mixing system of polynomial expressions.
Recall that by definition each $S\in\PE{F}{m}$ can be written in the form
\begin{equation}
\label{eq:PE-dec}
S(\alpha_{1},\dots,\alpha_{m}) = W^{(\alpha_{1},\dots,\alpha_{m-1})}(\alpha_{m}) \dots W^{\alpha_{1}}(\alpha_{2}) W(\alpha_{1}).
\end{equation}
\begin{definition}
\label{def:K-mixing-PE}
Let $K\leq F$ be a subgroup and $m\in\N$.
A set $\{S_{i}\}_{i=0}^{t} \subset\PE{F}{m}$ is called \emph{$K$-mixing} if $S_{0}\equiv 1_{G}$, the polynomial expressions $\{S_{i}\}$ are pairwise distinct, and for all $r$ and $i\neq j$ we have either $\forall\vec\alpha\in\Fin^{r}_{<}$ $W_{i}^{\vec\alpha}=W_{j}^{\vec\alpha}$ or $\forall\vec\alpha\in\Fin^{r}_{<}$ $(W_{i}^{\vec\alpha})\inv W_{j}^{\vec\alpha} \not\in K$.
\end{definition}
For $m=1$ this coincides with Definition~\ref{def:K-mixing}.
However, in general, this definition requires more than $\{W_{i}^{\vec\alpha}\}_{i}$ being (up to multiplicity) a $K$-mixing set in the sense of Definition~\ref{def:K-mixing} for every $\vec\alpha$.

\begin{theorem}[{cf.\ \cite[Theorem 4.12]{MR1692634}}]
\label{thm:mwm}
Let $Z\to Y$ be a $K$-primitive extension.
Then for every $m\geq 1$, every $K$-mixing set $\{S_{0},\dots,S_{t}\}\subset\PE{F}{m}$ and any $f_{0},\dots,f_{t}\in L^{\infty}(Z)$ we have wlog
\[
\IPlim_{\alpha_{1},\dots,\alpha_{m}} \norm[\Big]{ \E \big( \prod_{i=0}^{t} S_{i}(\alpha_{1},\dots,\alpha_{m})f_{i} \big| Y \big) - \prod_{i=0}^{t} S_{i}(\alpha_{1},\dots,\alpha_{m})\E(f_{i} | Y) } = 0.
\]
\end{theorem}
\begin{proof}
We induct on $m$.
The case $m=0$ is trivial since the product then consists only of one term.
Assume that the conclusion holds for $m$ and consider a $K$-mixing set of polynomial expressions in $m+1$ variables.
For brevity we write $\vec\alpha=(\alpha_{1},\dots,\alpha_{m})$ and $\alpha=\alpha_{m+1}$.
We may assume that $\norm{f_{i}}_{\infty}\leq 1$ for all $i$ and $\E(f_{i_{0}} | Y)=0$ for some $i_{0}$.

By Definition~\ref{def:K-mixing-PE} and with notation from \eqref{eq:PE-dec}, for every $\vec\alpha$ there exists a $K$-mixing set $\{V_{j}^{\vec\alpha}\} \subset F$ such that $W_{i}^{\vec\alpha}=V_{j_{i}}^{\vec\alpha}$, where the assignment $i\to j_{i}$ does not depend on $\vec\alpha$.
Let also
\[
A_{j} = \{ S(\cdot)=S_{i}(\cdot,\emptyset) : j_{i}=j\}.
\]
In view of the Milliken--Taylor theorem~\ref{thm:milliken-taylor} and by a diagonal argument, cf.\ \cite[Lemma 1.4]{MR833409}, it suffices to show that for every $\delta>0$ there exist $\vec\alpha<\alpha$ such that
\[
\norm[\Big]{ \E\big( \prod_{j}V_{j}^{\vec\alpha}(\alpha) (\prod_{S\in A_{j}}S(\vec\alpha)f_{S,j}) \big| Y \big) } \leq \delta
\]
provided that $\E(f_{S_{0},j_{0}} | Y)=0$ for some $j_{0},S_{0}$.
By the induction hypothesis there exists $\vec\alpha$ such that
\[
\norm[\Big]{ \E\big( \prod_{S\in A_{j_{0}}}S(\vec\alpha)f_{S,j_{0}} \big| Y \big) }
< \delta,
\]
since $A_{j_{0}}$ is a $K$-mixing set.
This implies
\[
\norm[\Big]{ \prod_{j}V_{j}^{\vec\alpha}(\alpha) \E\big(\prod_{S\in A_{j}}S(\vec\alpha)f_{S,j} \big| Y \big) } < \delta
\]
for all $\alpha>\vec\alpha$.
Since $\{V_{j}^{\vec\alpha}\}_{j}$ is a $K$-mixing set, Theorem~\ref{thm:wm} implies
\[
\IPlim_{\alpha} \norm[\Big]{ \E\big( \prod_{j}V_{j}^{\vec\alpha}(\alpha) \prod_{S\in A_{j}}S(\vec\alpha)f_{S,j} | Y \big) } \leq \delta.
\qedhere
\]
\end{proof}

\subsection{Lifting multiple recurrence to a primitive extension}
We are nearing our main result, a multiple recurrence theorem for polynomial expressions.
In order to guarantee the existence of the limits that we will encounter during its proof we have to restrict ourselves to a certain good subgroup of the group of polynomial expressions.
It will be shown later that this restriction can be removed, cf. Corollary~\ref{cor:finite-good}.
\begin{definition}
\label{def:good}
We call a group $\FE\leq\PE{\polyn}{\omega}$ \emph{good} if it has the following properties.
\begin{enumerate}
\item (Cardinality) $\FE$ is countable.
\item (Substitution) If $m\in\N$, $g\in \PE{\polyn}{m}\cap\FE$, and $\vec\beta\in\Fin^{m}_{<}$, then $g[\vec\beta]\in\FE$.
\item (Decomposition) If $K\leq F$ is a subgroup invariant under conjugation by constants and $\{S_{i}\}_{i=0}^{t} \subset\PE{F}{m}\cap\FE$ is a finite set with $S_{0}\equiv 1_{G}$,
then there exist finite sets $\{T_{k}\}_{k=0}^{v-1}\subset \PE{F}{m}\cap\FE$ and $\{R_{i}\}_{i=0}^{t}\subset \PE{K}{m}\cap\FE$ with $R_{0}=T_{0}\equiv 1_{G}$ such that $S_{i}=T_{k_{i}}R_{i}$ and for every sub-IP-ring $\Fin'\subset\Fin$ the set $\{T_{k}\}_{k=0}^{v-1}$ is wlog $K$-mixing.
\end{enumerate}
\end{definition}
The property of being good is hereditary in the sense that a group that is good with respect to some IP-ring is also good with respect to any sub-IP-ring.

Let $(X,\calA,\mu)$ be a regular measure space with a right action of $G$ by measure-preserving transformations.
Let also $\FE\leq\PE{F}{\omega}$ be a good group.
By Hindman's theorem~\ref{thm:hindman} we may wlog assume that
\[
\wIPlim_{\alpha} g(\alpha)f
\]
exists for every $g\in F$ and $f\in L^{2}(X)$.
By the Milliken--Taylor theorem~\ref{thm:milliken-taylor} we may wlog assume that the limit
\[
\IPlim_{\vec\alpha \in \Fin^{m}_{<}} \mu\left( \cap_{i=0}^{t} A S_{i}(\vec\alpha)\inv \right)
\]
exists for every $m\in\N$, every $A\in\calA$, and every finite set $\{S_{0},\dots,S_{t}\}\subset \PE{F}{m}\cap\FE$.
\begin{definition}
A factor $(X,\calA,\mu)\to (Y,\calB,\nu)$ is said to have the \emph{SZ (Szemer\'{e}di) property} if for every $B\in\calB$ with $\nu(B)>0$ and every set of polynomial expressions $\{S_{0}\equiv 1_{G},S_{1},\dots,S_{t}\} \subset\PE{F}{m}\cap\FE$ one has
\[
\IPlim_{\vec\alpha \in \Fin^{m}_{<}} \mu\left( \cap_{i=0}^{t} B S_{i}(\vec\alpha)\inv \right)
> a(B,m,\{S_{i}\}_{i}) > 0.
\]
\end{definition}
Note that unlike in \cite[Definition 5.1]{MR1692634} the constant depends not only on the number of polynomial expressions.
We cannot obtain more uniform results due to the lack of control on the number $w$ provided by Corollary~\ref{cor:pair-color}.

\begin{lemma}[{\cite[Proposition 5.2]{MR1692634}}]
\label{lem:maximal-SZ-factor}
For every separable regular measure-preserving system $X$ there exists a maximal factor that has the SZ property.
\end{lemma}

\begin{theorem}
\label{thm:SZ}
The identity factor $X\to X$ has the SZ property.
\end{theorem}
This is the central result of the present article and generalizes \cite[Theorem 1.3]{MR1692634}.
\begin{proof}
By Lemma~\ref{lem:maximal-SZ-factor} there exists a maximal factor $X\to Y$ with the SZ property.
Assume that $X\neq Y$, then by Theorem~\ref{thm:primitive-extension} wlog there exists a subgroup $K\leq F$ and a factor $X\to Z$ such that $(Z,\calC,\lambda)\to (Y,\calB,\nu)$ is a proper $K$-primitive extension.
We will show that $Z$ also has the SZ property, thereby contradicting maximality of $Y$.

Let $A\in\calC$ with $\lambda(A)>0$ and $\{S_{i}\}_{i=0}^{t} \subset\PE{F}{m}\cap\FE$ be a finite set with $S_{0}\equiv 1_{G}$.
We have to show
\begin{equation}
\label{eq:SZ-desired}
\IPlim_{\vec\alpha \in \Fin^{m}_{<}} \mu\left( \cap_{i=0}^{t} A S_{i}(\vec\alpha)\inv \right) > 0.
\end{equation}
By Lemma~\ref{lem:K-conj-inv} we may wlog assume that $K$ is an FVIP group and by Lemma~\ref{lem:ap-dense} that $\AP$ is dense in $L^{2}(Z)$.
Note that $\FE$ is still good with respect to the new IP-ring implied in the ``wlog'' notation.
Thus wlog we have a $K$-mixing set $\{W_{k}\}_{k=0}^{v-1}$ and polynomial expressions $R_{i}\in\PE{K}{m}\cap\FE$ with $R_{0}\equiv 1_{G}$ such that $S_{i}=W_{k_{i}}R_{i}$ for some $0\leq k_{i}<v$.

By Lemma~\ref{lem:indicator-ap} we may assume $1_{A}\in \AP$ after passing to a subset of $A$ if necessary.
There exist $c=c(\lambda(A))>0$ and a set $B\in\calB$ such that $\nu(B)>c$ and $\lambda_{y}(A)>c$ for every $y\in B$.
Pick $0 < \epsilon < \min(c/2,c^{v}/(4(t+1)))$.

By Corollary~\ref{cor:pair-color} there exist $N,w\in\N$ and
\[
\{L_{i},M_{i}\}_{i=1}^{w}\subset(\PE{K}{N}\cap\FE)\times(\PE{F}{N}\cap\FE)
\]
such that for every $l$-coloring of $\{L_{i},M_{i}\}$ there exists a number $a$ and sets $\beta_{1}<\dots<\beta_{m}\subset N$ such that the set $\{L_{a}R_{i}[\vec\beta],M_{a}W_{k}[\vec\beta]L_{a}\inv\}_{0\leq i\leq t,0\leq k< v}$ is monochrome (and in particular contained in the set $\{L_{i},M_{i}\}$).

Since $f=1_{A}\in \AP$ there exist functions $g_{1},\dots,g_{l}\in L^{2}(Z)$ and a set $D\in\calB$ such that $\nu(D)<\epsilon$ and for every $\delta>0$ and $T\in\PE{K}{N}\cap\FE$ there exists $\alpha_{0}$ such that for every $\alpha_{0}<\vec\alpha\in\Fin^{N}_{<}$ there exists a set $E=E(\vec\alpha)\in\calB$ with $\nu(E)<\delta$ such that for every $y\in (D\cup E)^{\complement}$ there exists $j$ such that $\norm{T(\vec\alpha)f-g_{j}}_{y}<\epsilon$.
Let $B'=B\cap D^{\complement}$, so that $\nu(B')>c/2$.

Let $Q=\abs{\Fin(N)^{m}_{<}}$ be the number of possible choices of $\vec\beta\in\Fin(N)^{m}_{<}$ and
\[
a_{1}:=a(B',N,\{1_{G}\}\cup \{M_{i}W_{k}[\vec\beta]\}_{1\leq i\leq w, k<v, \vec\beta\in\Fin^{m}_{<}(N)})>0.
\]

Using this with $\delta=a_{1}/2w^{2}$ and $T=L_{1},\dots,L_{w}$ we obtain wlog for every $\vec\alpha\in\Fin^{N}_{<}$ a set $E=E(\vec\alpha)\in\calB$ with $\nu(E)<a_{1}/2w$ such that for every $y\in (D\cup E)^{\complement}$ and every $i=1,\dots,w$ there exists $j=j(y,i)$ such that
\begin{equation}
\label{eq:ap}
\norm{ L_{i}(\vec\alpha)f-g_{j} }_{y} < \epsilon
\text{ for every } 1\leq i\leq w.
\end{equation}
By Theorem~\ref{thm:mwm} we may also wlog assume that for every $\vec\alpha\in\Fin^{N}_{<}$ we have
\begin{equation}
\label{eq:assumption-mixing}
\norm[\big]{ \E(\prod_{k<v}W_{k}(\vec\alpha)f | Y) - \prod_{k<v}W_{k}(\vec\alpha)\E(f | Y) } < c^{v} (a_{1}/2wQ)^{1/2}/4.
\end{equation}
Recall that we have to show \eqref{eq:SZ-desired}.
To this end it suffices to find $a(A,m,\{S_{i}\}_{0\leq i\leq t})$ such that for an arbitrary sub-IP-ring there exists $\vec\gamma\in\Fin^{m}_{<}$ with
\[
\mu\left( \cap_{i=0}^{t} A S_{i}(\vec\gamma)\inv \right)
> a(A,m,\{S_{i}\}_{0\leq i\leq t}) > 0,
\]
so fix a sub-IP-ring $\Fin$.
By definition of $a_{1}$ there exists a tuple $\vec\alpha\in\Fin^{N}_{<}$ (that will remain fixed) such that
\[
\nu\left( C_{0} \right) > a_{1}, \text{ where } C_{0}:=\cap_{i=1,\dots,w,k<v,\vec\beta} B' W_{k}[\vec\beta](\vec\alpha)\inv M_{i}(\vec\alpha)\inv.
\]
Let
\[
C := C_{0} \setminus \cup_{i=1}^{w} E(\vec\alpha) M_{i}(\vec\alpha)\inv,
\]
so that $\nu(C)>a_{1}/2$.
For every $y\in C$ consider an $l$-coloring of $\{L_{i},M_{i}\}_{i}$ given by $i\in [1,w] \mapsto j(yM_{i}(\vec\alpha),i)$ determined by \eqref{eq:ap}.
By the assumptions on $\{L_{i},M_{i}\}$ there exist $j(y)$, $a\in [1,w]$ and $\beta_{1}<\dots<\beta_{m}\subset N$ such that
\[
\norm{ L_{a}(\vec\alpha)R_{i}[\vec\beta](\vec\alpha)f-g_{j(y)} }_{yM_{a}(\vec\alpha)W_{k}(\vec\beta)L_{a}(\vec\alpha)\inv} < \epsilon
\text{ for every } 0\leq i\leq t,0\leq k<v.
\]
This can also be written as
\[
\norm{ W_{k}[\vec\beta](\vec\alpha)R_{i}[\vec\beta](\vec\alpha)f - W_{k}[\vec\beta](\vec\alpha)L_{a}(\vec\alpha)\inv g_{j(y)} }_{yM_{a}(\vec\alpha)} < \epsilon
\text{ for every } 0\leq i\leq t,0\leq k<v.
\]
Since this holds for every $i,k$ and we have $R_{0}\equiv 1_{G}$, this implies
\[
\norm{ (W_{k}R_{i})[\vec\beta](\vec\alpha)f-W_{k}[\vec\beta](\vec\alpha)f }_{y M_{a}(\vec\alpha)} < 2\epsilon.
\]
Passing to a subset $C'\subset C$ with measure at least $a_{1}/2wQ$, we may assume that $a$ and $\vec\beta$ do not depend on $y$.
Thus we obtain a set $B'':=C' M_{a}(\vec\alpha)$ of measure at least $a_{1}/2wQ$ and a tuple $(\gamma_{j}=\cup_{i\in\beta_{j}}\alpha_{i})_{j=1}^{m}$ such that
\[
\norm{ W_{k}R_{i}(\vec\gamma)f-W_{k}(\vec\gamma)f }_{y} < 2\epsilon
\]
for every $y\in B''$, $i$ and $k$.
Recall that $f$ is $\{0,1\}$-valued, so that
\[
\norm[\big]{ \prod_{i=0}^{t}S_{i}(\vec\gamma)f - \prod_{k<v}W_{k}(\vec\gamma)f }_{y}\\
=
\norm[\big]{ \prod_{i=0}^{t}W_{k_{i}}R_{i}(\vec\gamma)f - \prod_{i=0}^{t}W_{k_{i}}(\vec\gamma)f }_{y}\\
< 2(t+1)\epsilon
\]
for all $y\in B''$.
Moreover, since $B'' \subset \cap_{j} B' W_{j}(\vec\gamma)\inv$, one has
\[
\abs[\big]{ \prod_{k<v} W_{k}(\vec\gamma) \E(f | Y)(y) } \geq c^{v}
\]
for every $y\in B''$.
Therefore and by \eqref{eq:assumption-mixing} we obtain
\begin{align*}
\norm[\big]{ \prod_{i=0}^{t}S_{i}(\vec\gamma)f }
&\geq \norm[\big]{ \prod_{i=0}^{t}S_{i}(\vec\gamma)f }_{L^{2}(B'')}\\
&> \norm[\big]{ \prod_{k<v}W_{k}(\vec\gamma)f }_{L^{2}(B'')} - 2(t+1)\epsilon\nu(B'')^{1/2}\\
&\geq \norm[\big]{ \E(\prod_{k<v}W_{k}(\vec\gamma)f | Y) }_{L^{2}(B'')} - 2(t+1)\epsilon\nu(B'')^{1/2}\\
&\geq
\norm[\big]{ \prod_{k<v}W_{k}(\vec\gamma)\E(f | Y) }_{L^{2}(B'')}
-
\norm[\big]{ \E(\prod_{k<v}W_{k}(\vec\gamma)f | Y) - \prod_{k<v}W_{k}(\vec\gamma)\E(f | Y) }\\
&\qquad - 2(t+1)\epsilon\nu(B'')^{1/2}\\
&> c^{v} (a_{1}/2wQ)^{1/2}/4 =: a(A,m,\{S_{i}\}_{i})^{1/2}.
\qedhere
\end{align*}
\end{proof}

\subsection{Good groups of polynomial expressions}
As we have already mentioned, good groups are just technical vehicles.
The point is that we can perform all operations that we are interested in within a countable set of polynomial expressions, so that we can wlog assume the existence of all IP-limits that we encounter.

The only non-trivial property of good groups is the decomposition property.
However, the following lemma essentially shows that it is always satisfied.
\begin{proposition}
\label{prop:decomposition-mixing-compact}
Let $K\leq F$ be a subgroup that is invariant under conjugation by constants, $m\in\N$ and $\{S_{i}\}_{i=0}^{t} \subset\PE{F}{m}$ be any finite set with $S_{0}\equiv 1_{G}$.
Then there exists a set $\{T_{k}\}_{k=0}^{v}\subset\PE{F}{m}$ that is wlog $K$-mixing and decompositions $S_{i}=T_{k_{i}}R_{i}$ such that $R_{i}\in\PE{K}{m}$.
\end{proposition}
\begin{proof}
We argue by induction on $m$.
The claim is trivial for $m=0$.
Assume that it holds for $m$, we show its validity for $m+1$.
For brevity we write $\vec\alpha=(\alpha_{1},\dots,\alpha_{m})$, $\alpha=\alpha_{m+1}$.

Consider the maps $\tilde S_{i}(\vec\alpha) = S_{i}(\vec\alpha,\emptyset)$.
By the induction hypothesis there exists a set $\{\tilde T_{k}\}_{k=0}^{\tilde v}\subset\PE{F}{m}$ that is wlog $K$-mixing and decompositions $\tilde S_{i}=\tilde T_{k_{i}} \tilde R_{i}$ such that $\tilde R_{i}\in\PE{K}{m}$.
Then $S_{i}(\vec\alpha,\alpha)=W_{i}^{\vec\alpha}(\alpha) \tilde T_{k_{i}}(\vec\alpha) \tilde R_{i}(\vec\alpha)$.

Let $i<j$.
By the Milliken--Taylor Theorem~\ref{thm:milliken-taylor} we may wlog assume that either $(W_{i}^{\vec\alpha})\inv W_{j}^{\vec\alpha} \not\in K$ for all $\vec\alpha\in\Fin^{m}_{<}$ (in which case we do nothing) or $(W_{i}^{\vec\alpha})\inv W_{j}^{\vec\alpha} \in K$ for all $\vec\alpha\in\Fin^{m}_{<}$.
In the latter case we have $W_{j}^{\vec\alpha} = W_{i}^{\vec\alpha} R^{\vec\alpha}$ with some $R^{\vec\alpha} \in K$ and we can write
\[
S_{j}(\vec\alpha,\alpha) = W_{i}^{\vec\alpha}(\alpha) \tilde T_{k_{j}}(\vec\alpha)
\underbrace{(\tilde T_{k_{j}}(\vec\alpha)\inv R^{\vec\alpha} \tilde T_{k_{j}}(\vec\alpha))}_{\in K}(\alpha) \tilde R_{j}(\vec\alpha),
\quad
\vec\alpha\in\Fin^{m}_{<}.
\]
Doing this for all pairs $i<j$ we obtain the requested decomposition with the set $\{T_{k}\}$ consisting of all products $W_{i}^{\vec\alpha}\tilde T_{k_{j}}(\vec\alpha)$ that occur above.
\end{proof}

\begin{corollary}
\label{cor:finite-good}
Every finite subset of $\PE{F}{\omega}$ is wlog contained in a good subgroup of $\PE{\polyn}{\omega}$.
\end{corollary}
\begin{proof}
Since $F$ is a countable Noetherian group, it has at most countably many subgroups.
Moreover, each $\Fin^{m}_{<}$ is countable, and there are only countably many finite tuples in any countable set.
Hence we can use Proposition~\ref{prop:decomposition-mixing-compact} to obtain a countable descending chain of sub-IP-rings such that the decomposition property holds for each tuple for one of these sub-IP-rings.
The required sub-IP-ring is then obtained by a diagonal procedure, cf.\ \cite[Lemma 1.4]{MR833409}.
\end{proof}
Thus the good group is not really relevant for our multiple recurrence theorem, which we can now formulate as follows.
\begin{theorem}
\label{thm:SZ-general}
Let $G$ be a nilpotent group and $F\leq\polyn$ an FVIP group.
Consider a right measure-preserving action of $G$ on an arbitrary (not necessarily regular) probability space $(X,\calA,\mu)$.
Let $S_{0},\dots,S_{t}\in\PE{F}{m}$ be arbitrary polynomial expressions and $A\in\calA$ with $\mu(A)>0$.
Then there exists a sub-IP-ring $\Fin' \subset \Fin$ such that
\[
\IPlim_{\vec\alpha \in (\Fin')^{m}_{<}} \mu\left( \cap_{i=0}^{t} A S_{i}(\vec\alpha)\inv \right) > 0.
\]
\end{theorem}
\begin{proof}
We can assume $S_{0}\equiv 1_{G}$.
By Corollary~\ref{cor:finite-good} we may assume that $S_{0},\dots,S_{t}\in\FE$ for some good subgroup $\FE\leq\PE{\polyn}{m}$.
Then we can replace $G$ by a countable group that is generated by the union of ranges of elements of $\FE$.
Next, we can replace $\calA$ by a separable $G$-invariant $\sigma$-algebra generated by $A$.
Finally, we can assume that $X$ is regular and apply Theorem~\ref{thm:SZ}.
\end{proof}
Theorem~\ref{thm:SZ-intro} follows from Theorem~\ref{thm:SZ-general} and Lemma~\ref{lem:poly-fvip} with the filtration \eqref{eq:scalar-poly-filtration-2}, $d$ being the maximal degree of the generalized polynomials $p_{i,j}$.
By the Furstenberg correspondence principle we obtain the following combinatorial corollary.
\begin{corollary}
\label{cor:SZ-combinatorial}
Let $G$ be a finitely generated nilpotent group, $T_{1},\dots,T_{t}\in G$ and $p_{i,j}:\Z^{m}\to\Z$, $i=1,\dots,t$, $j=1,\dots,s$, be admissible generalized polynomials.
Then for every subset $E\subset G$ with positive upper Banach density the set
\[
\Big\{ \vec n\in\Z^{m} : \exists g\in G : g\prod_{i=1}^{t}T_{i}^{p_{i,j}(\vec n)}\in E, j=1,\dots,s \Big\}
\]
is FVIP* in $\Z^{m}$.
\end{corollary}
Observe that in Theorem~\ref{thm:gen-poly-FVIP} for (not necessarily admissible) generalized polynomials we can choose $n$ from a finite set that only depends on the generalized polynomial.
In view of this fact we have the following variant of Corollary~\ref{cor:SZ-combinatorial} for generalized polynomials.
\begin{corollary}
\label{cor:SZ-combinatorial-non-admissible}
Let $G$ be a finitely generated nilpotent group, $T_{1},\dots,T_{t}\in G$ and $p_{i,j}:\Z^{m}\to\Z$, $i=1,\dots,t$, $j=1,\dots,s$, be generalized polynomials.
Then there exist finite sets $S_{j}$, $j=1,\dots,s$, such that for every subset $E\subset G$ with positive upper Banach density the set
\[
\Big\{ \vec n\in\Z^{m} : \exists g\in G, \exists s_{j}\in S_{j} : gs_{j}\prod_{i=1}^{t}T_{i}^{p_{i,j}(\vec n)}\in E, j=1,\dots,s \Big\}
\]
is FVIP* in $\Z^{m}$.
\end{corollary}
Since every member set of an idempotent ultrafilter contains an IP set this implies a multidimensional version of \cite[Theorem 1.23]{MR2747062} that holds for every idempotent ultrafilter, see \cite[Remark 3.42]{MR2747062}.

\printbibliography
\end{document}